\pdfoutput=1

\documentclass[11pt,a4paper]{amsart}    
\usepackage{amsfonts,amsmath,latexsym,amssymb} 
\usepackage{amsthm}                
\usepackage[utf8]{inputenc}
\usepackage[T1]{fontenc}
\usepackage{mathtools}
\mathtoolsset{showonlyrefs=true} 
\usepackage{times}  
\usepackage{xspace}
\usepackage[final]{microtype} 
\usepackage{paralist} 
\usepackage{enumitem}
\usepackage[pdfusetitle,colorlinks,bookmarksopen=true,bookmarksnumbered=true,linkcolor=blue,urlcolor=blue,citecolor=blue]{hyperref}
\usepackage{lineno}
\usepackage{pgfplots}
\pgfplotsset{compat=newest}
\usetikzlibrary{patterns}
\usepgfplotslibrary{fillbetween}
\usepackage{tikzscale}
\usepackage{colortbl}
\usepackage{booktabs}
\usepackage{algorithm}
\usepackage[noend]{algpseudocode}
\usepackage[textsize=tiny,shadow]{todonotes}
 \newtheorem{theorem}{Theorem}[section]
\newtheorem{lemma}[theorem]{Lemma}

\newtheorem{prop}[theorem]{Proposition}
\newtheorem{proposition}[theorem]{Proposition}
\newtheorem{corollary}[theorem]{Corollary}
\newtheorem{hypo}[theorem]{Hypothesis}

\theoremstyle{definition}
\newtheorem{definition}[theorem]{Definition}
\newtheorem{example}[theorem]{Example}

\newtheorem{remark}[theorem]{Remark}

\newcommand{\todoinline}[1]{\todo[inline,size=\normalsize]{#1}}
\DeclareMathOperator{\dist}{\mathrm{dist}}
\DeclareMathOperator{\Ran}{\mathrm{Ran}}
\DeclareMathOperator{\Ker}{\mathrm{Ker}}
\DeclareMathOperator{\Kern}{\mathrm{Ker}}
\DeclareMathOperator{\Rank}{\mathrm{Rank}}
\DeclareMathOperator{\spec}{\mathrm{spec}}
\DeclareMathOperator*{\slim}{s-lim}
\DeclareMathOperator*{\wlim}{w-lim}
\newcommand{\cA}{\mathcal{A}}
\newcommand{\cB}{\mathcal{B}}
\newcommand{\cG}{\mathcal{G}}
\newcommand{\cL}{\mathcal{L}}
\newcommand{\cF}{\mathcal{F}}
\newcommand{\dd}{\mathrm d}
\newcommand{\ee}{\mathrm e}
\newcommand{\ii}{\mathrm i}
\newcommand{\eps}{\varepsilon}
\newcommand{\limeps}{\lim_{\eps \rightarrow 0^+} }
\newcommand{\abs}[1]{\left| #1 \right|} 
\newcommand{\scalarproduct}[2]{\langle #1, #2 \rangle}
\newcommand{\norm}[1]{\|#1\|}
\newcommand{\normbig}[1]{\left\|#1\right\|}
\newcommand{\sca}[1]{\langle #1 \rangle}
\newcommand{\scabig}[1]{\left< #1 \vphantom{#1} \vphantom{#1} \right>}
\newcommand{\NN}{\mathbb{N}}
\newcommand{\RR}{\mathbb{R}}
\newcommand{\CC}{\mathbb{C}}
\newcommand{\ZZ}{\mathbb{Z}}
\newcommand{\II}{\mathbb{I}}
\newcommand{\JJ}{\mathbb{J}}
\newcommand{\QQ}{\mathbb{Q}}
\newcommand{\Wcal}{\mathcal{W}}
\newcommand{\ess}{\mathrm{ess}}
\newcommand{\EE}{\mathsf{E}}
\newcommand{\one}{\mathbf{1}}
\newcommand{\boldI}{\mathbf{I}}
\newcommand{\Onetwo}{\mathds{1}_2}
\newcommand{\onematrix}{\mathds{1}}

\newcommand{\Lla}{L_{\lambda,\alpha}}
\newcommand{\Llam}{L_{\lambda,\alpha_m}}

\renewcommand{\P}{\ensuremath{\mathcal{P}}}
\newcommand{\fol}[2]{(#1_{#2})_{#2\in\NN}}
\newcommand{\foln}[2]{(#1)_{#2\in\NN}}
\newcommand{\sfol}[2]{\{#1\}_{#2\in\NN}}
\newcommand{\folz}[2]{(#1_{#2})_{#2\in\ZZ}}
\newcommand{\LL}{\ensuremath{\mathcal{L}}}
\newcommand{\plim}{\operatorname{\P-\lim}\limits}
\newcommand{\pli}{\stackrel{\P}{\rightarrow}}
\newcommand{\Ima}{\textnormal{Im\,}}
\DeclareMathOperator{\supp}{supp}
\DeclareMathOperator{\diam}{diam}

\newcommand{\red}[1]{{\color{red} #1}}
\newcommand{\blue}[1]{{\color{blue} #1}}
\newcommand{\gray}[1]{{\color{gray} #1}}

\newcommand{\eg}{\mbox{e.g.}\xspace}
\newcommand{\ie}{\mbox{i.e.}\xspace}

\newcommand{\energy}{E}
\newcommand{\period}{K}
\newcommand{\tr}{\operatorname{tr}}
\newcommand{\trans}{T}
\newcommand{\Lim}{\operatorname{Lim}}
\newcommand{\cfin}{\mathrm{c}_{00}}
\newenvironment{addmargin}[2][\empty]{\par
  \rightskip=#2\relax
  \ifx\empty#1\relax \leftskip=\rightskip
  \else \leftskip=#1\relax
  \fi}{\par}

\def\drawBand{\fill[bandSpecCol]}
\definecolor{pointSpecCol}{rgb}{0.0, 0.5, 0.0} 
\definecolor{bandSpecCol}{rgb}{0.52, 0.52, 0.51} 
 
\hypersetup{
pdfcreator={LaTeX with hyperref (master, d96c8b1a)},
pdfkeywords={Finite section method, limit operators, Jacobi operators},
pdfsubject={MSC Primary 65J10, 47B36; Secondary 47N50},
}

\begin{document}
\title{Finite Sections of Periodic Schrödinger Operators}

\author[Gabel, Gallaun, Gro{\ss}mann, Lindner, Ukena]{Fabian Gabel, Dennis Gallaun, Julian Gro{\ss}mann,\\Marko Lindner, Riko Ukena} 
\address{ 
Hamburg University of Technology,
Institute of Mathematics, 
Am Schwarzen\-berg-Campus 3, 
21073 Hamburg,
Germany} 
\email{\href{mailto:julian.grossmann@jp-g.de}{julian.grossmann@jp-g.de}} 
\email{\href{mailto:fabian.gabel@tuhh.de}{fabian.gabel@tuhh.de}} 
\email{\href{mailto:dennis.gallaun@tuhh.de}{dennis.gallaun@tuhh.de}} 
\email{\href{mailto:lindner@tuhh.de}{lindner@tuhh.de}} 
\email{\href{mailto:riko.ukena@tuhh.de}{riko.ukena@tuhh.de}} 


\begin{abstract}
We study discrete Schrödinger operators $H$ with periodic potentials as they are typically used to approximate aperiodic Schrödinger operators like the Fibonacci Hamiltonian. 
  We prove an efficient test for applicability of the finite section method, a procedure that approximates $H$ by growing finite square submatrices $H_n$. 
  For integer-valued potentials, we show that the finite section method is applicable as soon as $H$ is invertible. 
  This statement remains true for $\{0,\lambda\}$-valued potentials with fixed rational $\lambda$ and period less than nine as well as for arbitrary real-valued potentials of period two.
\end{abstract}

\maketitle

\section{Introduction}\label{sec:intro}
\subsection*{Discrete Schrödinger Operators}
If one discretises the one-dimensional Schrödinger operator $\Delta+a\cdot$ with potential $a$ by finite differences, one derives a so-called \emph{discrete Schrödinger operator} $H$, acting on a two-sided infinite sequence $x\colon\ZZ\to\CC$ via
\begin{equation}\label{eq:def-perschr-twosided}
(Hx)_k=x_{k-1}+x_{k+1}+v(k)x_{k}\,,\qquad k\in\ZZ\,.
\end{equation}
Typically, $x$ is an element of $\ell^p(\ZZ)$ with some $p\in[1,\infty]$, where $v$, again called the \emph{potential} of $H$, is essentially formed by samples of $a$.
We will assume that $a$ is bounded on $\RR$, whence $v$ is bounded on $\ZZ$, making $H$ a bounded linear operator on every space $\ell^p(\ZZ)$.
If $v$ is a periodic function, $H$ is called a \emph{periodic Schrödinger operator}.
In theoretical solid-state physics, periodic Schrödinger operators represent a standard model for periodic crystals.

Identifying $x$ with a two-sided infinite vector, $H$ acts as a two-sided infinite tridiagonal matrix $(H_{ij})_{i,j \in \ZZ}$ with the sequence $v$ on the main diagonal and ones on the diagonals immediately below and above. 
Note that $H$ is \emph{self-adjoint} on $\ell^2(\ZZ)$ if $v$ is real-valued. Otherwise, $H$ still is \emph{normal}.

The class of periodic discrete Schrödinger operators is very well-studied and several methods have been developed to understand its spectral properties. 
This includes \emph{Floquet--Bloch theory}, see, e.g., \cite[Chapter~7]{Teschl.2000} and \cite[Chapter~XIII.16]{Reed.1978}, the \emph{transfer matrix formalism}, see, e.g., \cite[Section~3]{Damanik.2021} and \cite{Last.1999}, 
and more recently and generally the application of the dynamical systems formalism, see \cite{Damanik.2017} and the references therein.
In this article, we study the spectrum and the approximation of periodic Schrödinger operators via the method of limit operators and the finite section method.

\subsection*{The Finite Section Method}
If~$H$ is invertible, to find the unique solution of the system~$Hx = b$, one often uses a truncation technique that replaces the original (infinite-dimensional) system with a sequence of finite linear systems: the \emph{finite section method (FSM)}.

Let $H_n$ correspond to the finite square submatrix $(H_{ij})_{i,j=-n}^n$, the so-called \emph{$n$-th finite section} of $H$.
We say that the FSM is \emph{applicable} to $H$ if, for sufficiently large $n$, the inverses $H_n^{-1}$ exist and approximate the inverse of $H$. 
By \emph{approximation} we mean that we first embed $H_n^{-1}$ into a two-sided infinite matrix with the same shape as $H$ and that $H_n^{-1}e_j\to H^{-1}e_j$ holds for all canonical unit vectors $e_j\in\ell^p(\ZZ)$. 
For $p<\infty$ this implies strong convergence $H_n^{-1}x\to H^{-1}x$ for all $x\in\ell^p(\ZZ)$. 
In particular, the solution of $Hx=b$ can be approximated by the solutions of corresponding finite-dimensional equations.

Some models also use half-line Schrödinger operators $H_+$ on $\ell^p(\ZZ_+)$, with $\ZZ_+ \coloneqq \{0,1,2,\dots\}$, which also can be seen as a restriction of $H$ from \eqref{eq:def-perschr-twosided} with a Dirichlet boundary condition, given by
\begin{equation}\label{eq:def-perschr-onesided}
    (H_+ x)_k=x_{k-1}+x_{k+1}+v(k)x_{k}\,,\quad k\in\ZZ_+\,, ~ \text{ where } ~x_{-1} = 0\,.
\end{equation}
This operator is then represented by a one-sided infinite matrix $(H_{ij})_{i,j=0}^\infty$. 
In this case, the \emph{$n$-th finite section} $H_n$ of $H_+$ is given by the matrix $(H_{ij})_{i,j=0}^n$.

A standard result, see, e.g., \cite{Lindner.2006,Rabinovich.2004}, shows that:
\begin{equation*} 
\begin{array}{c}
\textit{FSM applicable to $H$}
\end{array}
\ \iff\ \left\{
\begin{array}{l} 
\textit{(a)\, $H$ is invertible}\\
\textit{(b)\, all but finitely many $H_n$ are invertible}\\
\textit{(c)\, with uniformly bounded inverses $(H_n^{-1})$}\\
\end{array}
\right.
\end{equation*}
Note that the equivalence above also analogously holds for one-sided operators $H_+$ instead of $H$.

For some classes of operators, it turns out that the invertibility of $H$ (or $H_+$) alone is already sufficient for the applicability of the FSM, as it already implies the two other conditions. Of course, this simplifies FSM matters a lot. Let us call such an operator \emph{FSM-simple} -- either it has an applicable FSM, or it is not invertible. Consequently, if an FSM-simple operator is invertible, then the FSM is applicable.

\subsection*{Example: Fibonacci Hamiltonian}
The two-sided infinite discrete Schrödinger operator with potential 
\begin{equation*}
v(k)\ =\ \chi_{[1-\alpha,1)}(k\alpha \bmod 1)\, ,\qquad k\in\ZZ\,,
\end{equation*}
where $\alpha=\frac12(\sqrt5-1)$ is the golden ratio, is called the \emph{Fibonacci Hamiltonian}.
As~\cite{Lindner.2018} shows, this operator is FSM-simple.

The Fibonacci potential $v$ is a combinatorically very interesting mix of zeros and ones. Since $\alpha$ is irrational, it is not periodic but a so-called quasiperiodic sequence. The entries $v(k)$ with $k=-1,\dots,55$ look like this:
\begin{equation*}
1011010110110101101011011010110110101101011011010110101
\end{equation*}
Certain patterns are impossible, others reoccur in interesting ways that are not quite periodic. 
The patterns $101$ and $10$ repeat like the ones and zeros on the original level: we observe self-similarity.
The Fibonacci potential is, of course, well-studied and well-understood, see, e.g.,~\cite{Damanik.2000,Damanik.2007,Damanik.2016,Sueto.1995,Yessen.2011}. 
The proof in~\cite{Lindner.2018} that shows that $H$ is FSM-simple exploits a lot of this complexity and beauty of the Fibonacci Hamiltonian. 
At the same time, the proof heavily depends on the explicit structure of the Fibonacci Hamiltonian which raises the question of whether every discrete Schrödinger operator with a potential of zeros and ones is in fact FSM-simple.

\subsection*{Our Results}
In this paper, we prove sufficient conditions for a periodic Schrödinger operator to be FSM-simple.
In particular, all periodic Schrödinger operators with $\{0,1\}$-valued potential are FSM-simple.

\begin{theorem}\label{thm:integerFSM}
Let $ p \in [1,\infty]$ and $H$ be a periodic discrete Schrödinger operator on $\ell^p(\ZZ)$ given by \eqref{eq:def-perschr-twosided}, where $v$ is the potential.
\begin{enumerate}[label=(\alph*)]
\item\label{it:thmInt} If $v$ is integer-valued, then $H$ is FSM-simple.
\item\label{it:thmRat}  If the period is less than nine and $v$ only takes values in $\{0,\lambda\}$ with a fixed rational $\lambda$, then $H$ is FSM-simple.
\item\label{it:thmTwoPer}  If the period is two and the potential is real-valued, then $H$ is FSM-simple.
\end{enumerate}
If $H_+$ is a periodic discrete Schrödinger operator on $\ell^p(\ZZ_+)$ given by \eqref{eq:def-perschr-onesided}, then statements \ref{it:thmInt}--\ref{it:thmTwoPer} hold analogously for $H_+$ instead of $H$.
\end{theorem}

\begin{remark}\label{rem:12}
\phantomsection
\begin{enumerate}[label = (\alph*)]  
    \item\label{it:sharpExamples} The bound on the period length in Theorem~\ref{thm:integerFSM}\ref{it:thmRat} is optimal. 
      The operator $H$ with the $9$-periodic potential repeating the vector $\frac 12(1,1,0,1,0,1,0,1,1)$ is not FSM-simple, see Example~\ref{ex:9per} below.
The $5$-periodic potential repeating $\frac 1{\sqrt2}(1,1,0,1,0)$ shows that Theorem~\ref{thm:integerFSM}\ref{it:thmRat} fails if we drop rationality of $\lambda$, see Example~\ref{ex:5per} below.
\item\label{it:periodThree}  The $3$-periodic potential repeating $(2,\frac 12,\frac 12)$ shows that Theorem~\ref{thm:integerFSM}\ref{it:thmTwoPer} cannot include period three, see Example~\ref{ex:3per} below.
\end{enumerate}
\end{remark}

For $K$-periodic Schrödinger operators $H$ on $\ell^p(\ZZ)$ with real-valued potential $v$, the invertibility is particularly easy to check, see, e.g.,~\cite{Puelz.2014}: $H$ is invertible if and only if any of the following matrices has its trace outside $[-2,2]$:
\begin{align}\label{eq:M^j}
\quad M^{(j)}&=
\begin{pmatrix} -v(K-1+j)&-1\\1&0\end{pmatrix}\cdots 
\begin{pmatrix} -v(1+j)&-1\\1&0\end{pmatrix}
    \begin{pmatrix} -v(j)&-1\\1&0\end{pmatrix}\\[1em]
\label{eq:N^j}
  \widetilde{M}^{(j)}&=
\begin{pmatrix} -v(j)&-1\\1&0\end{pmatrix}
\begin{pmatrix} -v(1+j)&-1\\1&0\end{pmatrix}\cdots 
\begin{pmatrix} -v(K-1+j)&-1\\1&0\end{pmatrix}
\end{align}
with $j\in\ZZ$. 
If $H$ is FSM-simple, then that already settles the applicability of the FSM.
However, for periodic Schrödinger operators that are possibly not FSM-simple, 
in addition to the aforementioned trace condition, one has to check whether
\begin{equation}\label{eq:checkM}
M_{2,1}\ne 0\qquad\text{or}\qquad |M_{1,1}|>1\,,
\end{equation}
where $M$ runs through the set of matrices~\eqref{eq:M^j} and \eqref{eq:N^j}.

\begin{theorem}\label{thm:M21_M11}
  Let $ p \in [1,\infty]$, $H$ be a $K$-periodic discrete Schrödinger operator on $\ell^p(\ZZ)$ 
  with real-valued potential $v$ defined by \eqref{eq:def-perschr-twosided},
	and $H_+$ be the restriction on $\ell^p(\ZZ_+)$ defined by~\eqref{eq:def-perschr-onesided}. 
  In addition, let $M^{(j)}$ and $\widetilde{M}^{(j)}$ be given by \eqref{eq:M^j} and \eqref{eq:N^j}. Then the following holds:
\begin{enumerate}[label=(\alph*)]
    \item\label{it:thmTwoSided}
The FSM is applicable to $H$ if and only if 
$|\tr(M^{(0)})|>2$ and
\begin{equation*}
    M^{(0)},\dots,M^{(K-1)} \text{ and } \widetilde{M}^{(0)},\dots,\widetilde{M}^{(K-1)}
    \text{are subject to \eqref{eq:checkM}}.
\end{equation*}
\item\label{it:thmOneSided} 
The FSM is applicable to $H_+$ if and only if
$|\tr(M^{(0)})|>2$ and
\begin{equation*}
M^{(0)} \text{ and } \widetilde{M}^{(0)},\ldots,\widetilde{M}^{(K-1)}
\text{are subject to \eqref{eq:checkM}}\,.
\end{equation*}
\end{enumerate}
In particular, the FSM is applicable to $H_+$ if it is applicable to its periodic extension $H$ on $\ell^p(\ZZ)$. Moreover, $H_+$ is FSM-simple if $H$ is FSM-simple.
\end{theorem}

\begin{remark}\label{rem:14}
The $9$-periodic Schrödinger operator $H$ with potential consisting of 
repeated $\frac 1{\sqrt 2}(1, 1, 1, 0, 1, 1, 0, 1, 0)$ is an example where the FSM is applicable to $H_+$ but not to $H$, see Example~\ref{ex:onesided_only} below.
\end{remark}

\section{The FSM: Operators and Tools}\label{sec:fsm_intro}
\subsection{Band Operators and the FSM}
A discrete Schrödinger operator~$H$ can analogously be described by a two-sided infinite tridiagonal matrix~$A = (a_{ij})_{i,j \in \ZZ}$ which is defined by
	\begin{equation}\label{eq:twoSidedInfiniteMatrixRep}
		A = 
	      \begin{pmatrix}
              \ddots & \ddots  &        &      \\	
		\ddots  & 	v(0)      & 1   &        &        &            \\
		&	1     & v(1)    &  1   &        &              \\
         &   &  1    & v(2)    & \ddots &                 \\
           &        &  & \ddots & \ddots         
		\end{pmatrix} \,.
	\end{equation}

The so-called \emph{finite section method (FSM) for~$H$ or $A$} considers the sequence of finite submatrices 
\begin{equation}\label{eq:submatricesFSM}
A_n = (a_{ij})_{i,j=-n}^{n},\quad n\in\NN \,,
\end{equation}
and asks the following: 
\vskip1ex
\begin{addmargin}[2em]{2em}
  \noindent\emph{Are the matrices $A$ and $A_n$ invertible for all sufficiently large $n$ and are their inverses (after embedding them into a two-sided infinite matrix) strongly convergent to the inverse of $A$?}
\end{addmargin}
\vskip1ex
In the case of a positive answer, we call the FSM (for $A$) \emph{applicable}.
This approximation of $A^{-1}$ can be used for solving equations $Ax=b$ approximately via the solutions of growing finite systems $A_n$.
First rigorous treatments of this natural approximation can be found in~\cite{Baxter.1963, Gohberg.1974}. 
In case the FSM is not applicable in the above sense, it may still be possible to establish an applicability result by passing to suitable 
modifications $A_n = (a_{ij})_{i,j=l_n}^{r_n}$ with well-chosen sequences $l_n\to-\infty$ and $r_n\to\infty$, see~\cite{Lindner.2010, Rabinovich.2008}.

Neither the FSM nor the following tools and results in Section~\ref{sec:fsm_intro} require self-adjointness (unless we explicitly say so).

Operators like the Schrödinger operator $H$ whose infinite matrix representation only exhibits finitely many non-zero diagonals are a well-known subject of investigation regarding the applicability of the FSM.
Operators of this type are summarised in the class of \emph{band operators}, which we introduce next. 

\begin{definition}[Band-Width and Band Operator]\label{def:band-width} 
A finite sum 
\begin{equation*}
A\ =\ \sum_{k=-\omega}^\omega M_{a^{(k)}}S^k
\end{equation*} 
of products of
multiplication operators $(M_{a^{(k)}}x)_n=a^{(k)}_nx_n$ with $a^{(k)}\in\ell^\infty(\ZZ)$
and powers of the shift operator, $(S x)_n = x_{n - 1}$, is called \emph{band operator} with 
\emph{band-width} $\omega\in\NN$.
\end{definition}

By definition, a band operator $A$ acts as a bounded operator on every $\ell^p(\ZZ)$ with $p\in[1,\infty]$. In the following, we will identify an operator $A$ on $\ell^p(\ZZ)$ with its usual \emph{matrix representation} $(a_{ij})_{i,j \in \ZZ}$ with respect to the canonical basis in $\ell^p(\ZZ)$. 
For band operators with band-width $\omega$, this matrix is only supported on the $k$-th diagonals with $-\omega\le k\le \omega$.

\begin{remark}\label{rem:band-width}
Whenever necessary, we consider for some index set $\II \subset \ZZ$ the space $\ell^p(\II)$ as a subspace of $\ell^p(\ZZ)$.
  Then the terms in Definition~\ref{def:band-width} naturally carry over to operators $A$ on $\ell^p(\II)$ by associating them with the restriction onto the subspace $\ell^p(\II)$ of their canonical extension by zero.
\end{remark}

Here is one fundamental fact about band operators:

\begin{proposition}\label{prop:kurbatov}
    Let $B$ be a band operator on $\ell^p(\II)$ for some $\II \subset \ZZ$ and $p \in [1,\infty]$.
    If~$B$ is invertible, then $B$ is also invertible as an operator on $\ell^q(\II)$ for all $q \in [1,\infty]$.
\end{proposition}
\begin{proof}
	First, we identify the Banach space $\ell^p(\ZZ)$ with the $p$-direct sum $\ell^p(\JJ) \oplus \ell^p(\II)$
	where $\JJ \coloneqq  \ZZ \setminus \II$
	and the norm is given by 
        \begin{equation*}
          \norm{x \oplus y} \coloneqq  
          \begin{cases}
            \big( \norm{x}^p_{\ell^p(\JJ)} 
            + \norm{y}^p_{\ell^p(\II)})^{\frac{1}{p}} &\quad\text{if } p < \infty\,,\\[0.5em]
            \max \big\{
            \norm{x}_{\ell^\infty(\JJ)} , 
            \norm{y}_{\ell^\infty(\II)}  \big\}  &\quad\text{if } p  = \infty \,.
          \end{cases}
        \end{equation*}
    Now, assume that~$B$ is invertible on $\ell^p(\II)$.
    We extend the operator~$B$ to an invertible operator~$A$ on~$\ell^p(\ZZ)$ via
    \begin{equation*}
        A \coloneqq  
        \begin{pmatrix}
            \one_{\JJ} & 0 \\
            0 & B
        \end{pmatrix}
    \end{equation*}
    with respect to the direct decomposition $\ell^p(\ZZ)  = \ell^p(\JJ) \oplus \ell^p(\II)$.
    Here, $\one_{\JJ}$ denotes the identity on $\ell^p(\JJ)$.
    Note that, as a band operator, $A$ is an element of the so-called \emph{Wiener algebra} $\Wcal$, see, e.g.,~\cite[Section 3.7.3]{Lindner.2009}.
    Due to~\cite[5.2.10]{Kurbatov.1999}, the algebra $\Wcal$ is closed under taking inverses; a concise proof of this fact can also be found in~\cite[Corollary~2.5.4]{Rabinovich.2004}.
    As an element of $\Wcal$, the inverse $A^{-1}$ acts boundedly on every space $\ell^q(\ZZ)$, $q \in [1,\infty]$.
    Consequently, the operator~$A$ is invertible on all spaces~$\ell^q(\ZZ)$, $q \in [1,\infty]$.
    Moreover, we have
    \begin{equation*}
        A^{-1} 
        = 
        \begin{pmatrix}
            \one_{\JJ} & 0 \\
            0 & B^{-1}
        \end{pmatrix} \,.
    \end{equation*}
    Finally, we conclude that $B^{-1}$ is a bounded operator on $\ell^q(\II)$ for all $q \in [1,\infty]$.
\end{proof}

Back to the FSM: assuming invertibility of $A$ on $\ell^p(\ZZ)$, the applicability of the FSM is equivalent to the uniform boundedness of the inverses $A_n^{-1}$, a concept that is also known as \emph{stability}.
\begin{definition}[Stability]\label{def:stability}
    A sequence of operators $(A_n)_{n\in\NN}$ defined on $\ell^p(\ZZ)$ or $\ell^p(\ZZ_+)$ for $p \in [1,\infty]$ is called \textit{stable} if there exists~$n_0\in\NN$ such that for all $n\ge n_0$, the operators $A_n$ are invertible and their inverses are uniformly bounded.
\end{definition}
The basic result connecting the notions of applicability and stability is known as \emph{Polski's~theorem}, cf.~\cite[Theorem~1.4]{Hagen.2001},
and it says that $(A_n)$ is applicable to $A$ if and only if it is stable and $A$ is invertible.

Now, stability, and hence applicability, of the FSM sequence $(A_n)$ in~\eqref{eq:submatricesFSM} is closely connected to the following entrywise limits
\begin{equation}\label{eq:ass_matrix}
(a_{i+l_n,j+l_n})_{i,j=0}^\infty\ \xrightarrow{n\rightarrow \infty}\ L_+\quad\textrm{and}\quad (a_{i+r_n,j+r_n})_{i,j=-\infty}^0\ \xrightarrow{n\rightarrow \infty}  R_-
\end{equation}
of one-sided infinite submatrices of $A$, where we 
consider sequences $(l_n)$ and $(r_n)$ with $\lim_{n \to \infty} l_n = -\infty$ and $\lim_{n \to \infty} r_n = \infty$ such that the limits in~\eqref{eq:ass_matrix} exist. 
The following result summarises the aforementioned connections and is a standard outcome regarding the applicability of the FSM for band operators. 

\begin{lemma}[{\cite[Lemma 1.2]{Chandler-WildeLindner.2016}, \cite[Theorem 2.3]{Rabinovich.2008}}]\label{lem:FSMappl}
\begin{enumerate}[label=(\alph*)]
 \item\label{it:lemApplTwoSided} Let $A$ be a band operator on $\ell^p(\ZZ)$ for $p\in[1,\infty]$.  
Then the following are equivalent:
\begin{enumerate}[label=(\roman*)]
  \item\label{it:FSMAppl} The FSM is applicable to $A$.
  \item\label{it:seqStable} The sequence $(A_n)$ with $A_n = (a_{ij})_{i,j=-n}^{n}$ is stable.
\item\label{it:AandLimOpsInv} The operator $A$ and the limits $L_+$ and $R_-$ from~\eqref{eq:ass_matrix} are invertible for all suitable sequences $(l_n)$ and $(r_n)$. 
\end{enumerate}

\item\label{it:lemApplOneSided} Let $A_+$ be a band operator on $\ell^p(\ZZ_+)$ for $p\in[1,\infty]$, where $\ZZ_+ \coloneqq \{k\in\ZZ:k\ge 0\}$.  
Then the following are equivalent:
\begin{enumerate}[label=(\roman*)]
\setcounter{enumii}{3}
  \item\label{it:FSMAppl+} The FSM is applicable to $A_+$.
  \item\label{it:seqStable+} The sequence $(A_n)$ with $A_n = (a_{ij})_{i,j=0}^{n}$ is stable.
\item\label{it:AandLimOpsInv+} The operator $A_+$ and the limits $R_-$ from~\eqref{eq:ass_matrix} are invertible for all suitable sequences $(r_n)$. 
\end{enumerate}
\end{enumerate}
\end{lemma}

\subsection{Limit Operators}\label{sec:invAndLimOps}
  We now focus on an operator-theoretical tool in order to reformulate conditions~\ref{it:AandLimOpsInv} and \ref{it:AandLimOpsInv+} of Lemma~\ref{lem:FSMappl}: the so-called \emph{limit operators}, cf.~\cite{Lindner.2006,Rabinovich.1998,Rabinovich.2004}.
In the following, let $\ZZ_- \coloneqq -\ZZ_+$. Note that both sets, $\ZZ_-$ and $\ZZ_+$, include zero.

\begin{definition}[Limit Operators and Compressions]\label{def:limOps}
  Let $A$ be a band operator on $\ell^p(\mathbb{I})$ for $p\in[1,\infty]$ and $\mathbb{I}\in\{\ZZ_+,\ZZ_-,\ZZ\}$. 
  An operator $B\in\ell^p(\ZZ)$ with matrix representation $(b_{ij})_{i,j \in \ZZ}$ is called \emph{a limit operator of $A$} if there is a sequence  $h=(h_n)_{n\in\NN} \subset \ZZ$ with 
  $\lim\limits_{n\rightarrow\infty} |{h}_n| = \infty$ and
  \[ a_{i+h_n,j+h_n}\xrightarrow{n \rightarrow \infty} b_{ij} \]  
  for all $i,j\in\ZZ$.
  In this case, we also write $A_h  \coloneqq B$ and say that $h$ is the \emph{corresponding sequence to $B$}. 
  We denote the set of all limit operators of $A$ by $\Lim(A)$.
  For~$A_h\in\Lim(A)$, we write~$A_h\in\Lim_+(A)$ or $A_h\in\Lim_-(A)$ if the corresponding sequence $h$ tends to $+\infty$ or $-\infty$, respectively.
  Moreover, if $\II = \ZZ$, we call both operators
  \begin{equation*}
   A_+ \coloneqq (a_{ij})_{i,j = 0}^\infty \quad\text{and}\quad A_- \coloneqq (a_{ij})_{i,j = -\infty}^0
  \end{equation*}
  \emph{(one-sided) compressions of} $A$. Then we have the identities $\Lim_\pm(A) = \Lim(A_\pm)$.
\end{definition} 

In fact, the operators $L_+$ and $R_-$ from~\eqref{eq:ass_matrix} correspond to the one-sided compressions of the limit operators $A_{(l_n)}$ and $A_{(r_n)}$.
Evidently, a limit operator $A_h$ of a band operator $A$ is again a band operator with the same band-width.

We now reformulate Lemma~\ref{lem:FSMappl} in the language of limit operators.

\begin{prop}[{\cite[Lemma 1.2]{Chandler-WildeLindner.2016}, \cite[Theorem 2.3]{Rabinovich.2008}}]\label{prop:FSMCondition}
Let $A$ be a band operator on $\ell^p(\ZZ)$ for $p\in[1,\infty]$.
Then the FSM for~$A$ is applicable if and only if
\begin{enumerate}[label=(\alph*)]
\item\label{it:AinvLp} $A$ is invertible on $\ell^p(\ZZ)$,
\item for all $R\in\Lim_+(A)$ the compressions $R_-$ are invertible on $\ell^p(\ZZ_-)$, and 
\item for all $L\in\Lim_-(A)$ the compressions $L_+$ are invertible on $\ell^p(\ZZ_+)$.
\end{enumerate}

Now let $A_+$ be a band operator on $\ell^p(\ZZ_+)$ for $p\in[1,\infty]$.
Then the FSM for~$A_+$ is applicable if and only if 
\begin{enumerate}[label=(\alph*)]
\setcounter{enumi}{3}
\item\label{it:AinvLp+} $A_+$ is invertible on $\ell^p(\ZZ_+)$ and
\item for all $R\in\Lim_+(A)$ the compressions $R_-$ are invertible on $\ell^p(\ZZ_-)$.
\end{enumerate}
\end{prop}

The restriction to a particular choice of $p$ in Proposition~\ref{prop:FSMCondition} can be dropped, thanks to Proposition~\ref{prop:kurbatov}.
In subsequent parts of this article, this result will allow us to fall back onto the Hilbert space case~$p = 2$ whenever needed.

Recall that a bounded linear operator $A \colon X \rightarrow Y$
between Banach spaces $X$ and $Y$
is called a \emph{Fredholm operator} if
the kernel $\Ker(A)$
and the cokernel $Y/\Ran(A)$ are finite-dimensional.
Furthermore, we define the \emph{essential spectrum of $A$} by
\begin{equation*}
  \sigma_\ess(A) \coloneqq  \{ E \in \CC  : A - E  \text{ is not a Fredholm operator} \}\,.
\end{equation*}

The next lemma establishes the equivalence of the Fredholmness of a band operator $A$, the invertibility of its limit operators, see~\cite{Lindner.2014,Rabinovich.1998}, or, alternatively, their injectivity on the space $\ell^\infty(\mathbb{Z})$, see~\cite{Wilde-Lindner.2008,Chandler-WildeLindner.2011}.
The lemma previously appeared in this form in~\cite{Lindner.2018}.

\begin{lemma}[{\cite[Lemma 2.6]{Lindner.2018}}]\label{lem:fundamental}
    Let $\II \in \{\ZZ, \ZZ_+, \ZZ_-, \NN\}$ and $p \in [1,\infty]$.
    For a band operator~$A$, the following are equivalent:
    \begin{enumerate}[label=(\roman*)]
      \item\label{it:AFredholm} $A$ is a Fredholm operator on $\ell^p(\II)$.
      \item\label{it:allLimOpsInvertible} All limit operators of $A$ are invertible on $\ell^p(\ZZ)$.
      \item\label{it:allLimOpsInjective} All limit operators of $A$ are injective on $\ell^\infty(\ZZ)$.
    \end{enumerate}
\end{lemma}

\begin{remark}
Actually, it is Lemma~\ref{lem:fundamental}, whose first versions (with proof) go back to \cite{Rabinovich.1998} and somehow finalise in \cite{Wilde-Lindner.2008,Lindner.2014}, where limit operators naturally enter the scene, and they only reappear in theorems on applicability and stability of operator sequences because the latter is equivalent to Fredholmness of an associated operator.
\end{remark}

Here is a simple consequence of Lemma~\ref{lem:fundamental} about spectra and essential spectra of band operators and their limit operators, see also~\cite[Corollary 12]{Lindner.2014}:

\begin{proposition}\label{prop:essSpecEqualsSpecnew}
	Let $A$ be a band operator on $\ell^p(\II)$ with $\II \in \{\ZZ, \ZZ_+, \ZZ_-, \NN\}$.
	Then
	\begin{equation*}
		\sigma_\ess(A) = \bigcup_{B \in \Lim(A)} \sigma(B) \, .
	\end{equation*}
	In particular, we have:
		\begin{enumerate}[label=(\alph*)]
			\item\label{it:equal_spectra_ess} If $\sigma(B) = \sigma(B^\prime)$ for all $B, B^\prime \in \Lim(A)$, then $\sigma_\ess(A) = \sigma(B)$.
			\item\label{it:spectra_of_limops} If $B\in\Lim(A)$, then $\sigma(B)\subset\sigma_\ess(A)\subset\sigma(A)$.
			\item\label{it:spec_is_ess} If $\II=\ZZ$ and $A \in \Lim(A)$, then $\sigma(A) = \sigma_\ess(A)$.
		\end{enumerate}
\end{proposition}

\begin{proof}
	By definition,
	$\energy \in \sigma_\ess(A)$ if and only if $A-\energy$ is not Fredholm.
	By the equivalence~\ref{it:AFredholm}$\Leftrightarrow$\ref{it:allLimOpsInvertible}
	in Lemma~\ref{lem:fundamental}, this is equivalent to $B-\energy$ not being invertible, i.e.,\ $\energy \in \sigma(B)$, for some $B \in \Lim(A)$.
\end{proof}

Lemma~\ref{lem:fundamental} is particularly useful when dealing with a band operator $A$ that additionally has the property~$A \in \Lim(A)$. 
All subsequent examples of Schrödinger operators will have this property, which is sometimes referred to as~\emph{self-similarity}, cf.~\cite{Wilde-Lindner.2008}.
The combination of Proposition~\ref{prop:kurbatov} with Lemma~\ref{lem:fundamental} leads to the following results that allow us to translate the invertibility problem of an operator into an injectivity problem of its limit operators.

\begin{corollary}\label{cor:fundamental}
    Let $A$ be a band operator with $A \in \Lim(A)$.
    Then the following are equivalent:
    \begin{enumerate}[label=(\roman*)]
        \item\label{it:LimInj} All limit operators of $A$ are injective on $\ell^\infty(\ZZ)$.
        \item\label{it:AInvAll} $A$ is invertible on~$\ell^p(\ZZ)$ for all  $p \in [1,\infty]$.
        \item\label{it:AInvSome} $A$ is invertible on~$\ell^p(\ZZ)$ for some $p \in [1,\infty]$.
    \end{enumerate}
\end{corollary}

\begin{proof}
  The equivalence~\ref{it:AInvAll}$\Leftrightarrow$\ref{it:AInvSome}
  is immediately given by Proposition~\ref{prop:kurbatov}.
  Since $A \in \Lim(A)$, the implication~\ref{it:LimInj}$\Rightarrow$\ref{it:AInvSome}
  follows from Lemma~\ref{lem:fundamental}.
  Using that an invertible operator is a Fredholm operator, we also have~\ref{it:AInvSome}$\Rightarrow$\ref{it:LimInj} by Lemma~\ref{lem:fundamental}.
\end{proof}

Note that Corollary~\ref{cor:fundamental} only handles operators that are defined on $\ell^p(\ZZ)$.
However, Lemma~\ref{lem:FSMappl} also relies on the invertibility of one-sided compressions.
Therefore, the next corollary gives the corresponding result.
For this, we consider only the case
$p = 2$ because there we are able to use the \emph{Hilbert space adjoint} $A^*$
for an operator~$A$ on the Hilbert space $\ell^2(\II)$ with $\II \subset \ZZ$.

\begin{corollary}\label{cor:injImpliesInvertibleForLimOps}
    Let $B$ be a self-adjoint invertible band operator on $\ell^2(\ZZ)$.
    \begin{enumerate}[label=(\alph*)]
      \item\label{it:B-inj} 
        If the compression $B_-$ is injective on $\ell^\infty(\ZZ_-)$, then $B_-$ is invertible on~$\ell^p(\ZZ_-)$ for all $p \in [1,\infty]$.
      \item\label{it:B+inj} 
        If the compression $B_+$ is injective on $\ell^\infty(\ZZ_+)$, then $B_+$ is invertible on~$\ell^p(\ZZ_+)$ for all $p \in [1,\infty]$.
    \end{enumerate}
\end{corollary}

\begin{proof}
  We only prove~\ref{it:B-inj} since the proof of~\ref{it:B+inj} works completely analogously.
    For the proof of~\ref{it:B-inj}, we note that, due to Proposition~\ref{prop:kurbatov}, it suffices to consider $p = 2$. 
     As $B_-$ is injective on $\ell^\infty(\ZZ_-)$, it is also injective on the subset $\ell^2(\ZZ_-) \subset \ell^\infty(\ZZ_-)$.
    We will show that $B_-$ has a dense and closed range making $B_-$ also surjective. 
    
    \emph{The range of $B_-$ is dense in $\ell^2(\ZZ_-)$:}   
    Since $B$ is self-adjoint, its compression $B_-$ is also self-adjoint.
    Therefore, the adjoint $(B_-)^*$ is also injective on $\ell^2(\ZZ_-)$ which implies that the range of $B_-$ is dense in $\ell^2(\ZZ_-)$.

  \emph{The range of $B_-$ is closed in $\ell^2(\ZZ_-)$:}
    Since $B$ is invertible, it is in particular Fredholm on $\ell^2(\ZZ)$. 
    Consequently, Lemma~\ref{lem:fundamental}\ref{it:AFredholm}$\Rightarrow$\ref{it:allLimOpsInvertible} gives that all operators in $\Lim(B) \supset \Lim(B_-)$ are invertible on $\ell^2(\ZZ)$.
    As all limit operators of $B_-$ are invertible on $\ell^2(\ZZ)$,
    the implication~\ref{it:allLimOpsInvertible}$\Rightarrow$\ref{it:AFredholm} in Lemma~\ref{lem:fundamental} gives that $B_-$ is Fredholm. 
    In particular, $B_-$ has a closed range. \qedhere
\end{proof}

\section{Periodic Schrödinger Operators}\label{sec:periodSchr}
In this section, we will analyse periodic Schrödinger operators
and always assume the hypothesis below.

\begin{hypo}\label{hypo:PeriodicH}
 Let $ p \in [1,\infty]$
 and $H : \ell^p(\ZZ) \rightarrow \ell^p(\ZZ)$
 be given by
 	\begin{equation}
 	\label{eq:recursionPeriodicH}
 	(H x)_n =  x_{n+1} +   x_{n-1} + v(n) x_n \, , \quad n \in \ZZ \,
 	\end{equation}
 with the real-valued potential $v \colon \ZZ \to \RR$.
 Assume in addition
 that $v$ is periodic with period $\period$,
 which means $\period \in \NN$
 and $v(n+\period) = v(n)$ for all $n \in \ZZ$.
\end{hypo}

The potential $v$ as well as the operator $H$  from above are simply called \emph{$K$-periodic}. Many facts about such periodic Schrödinger operators
are known, see, e.g.,~\cite{Embree.2017,Puelz.2014,Teschl.2000}, such that this section mainly focuses on
the interpretation of these facts in the framework of limit operator theory. 

\subsection{Trace Condition for the Spectrum of Periodic Schrödinger Operators}\label{sec:traceConditionPeriodic}
If we choose 
an energy~$\energy \in \RR$ and a vector $x \in \ker(H-\energy)$, 
relation~\eqref{eq:recursionPeriodicH} yields the eigenvalue equation 
\begin{equation*}
  0 = ((H - E)x)_n = x_{n + 1} + x_{n - 1} + (v(n) - E) x_n\, , \quad n \in \ZZ \,.
\end{equation*}
This scalar three-term recurrence can be written as the vector-valued two-term recursion
	\begin{equation}
		\label{eq:Jac_recursion_formula2}
				\begin{pmatrix}
					  x_{n+1} \\
						x_n
				\end{pmatrix}	
 				=
 			\begin{pmatrix}
 		 		 \energy-v(n)  & -1 \\
 	   			  1 & 0
 			\end{pmatrix}
  			\begin{pmatrix}
 				 x_n \\
 				x_{n-1} 
 			\end{pmatrix}
 			~\text{ for all } n \in \ZZ.
	\end{equation}
The $2 \times 2$ matrix in~\eqref{eq:Jac_recursion_formula2} is called \emph{transfer matrix} and here denoted 
for each $n \in \ZZ$ by:
	\begin{equation}\label{eq:Transfer-Definition1}
		\trans(n,\energy)  \coloneqq 
	\begin{pmatrix}
 		\energy-v(n)   & -1 \\
 	     1 & 0
 	\end{pmatrix}
 	\,.
	\end{equation}
          
        The transfer matrices lie in the \emph{symplectic group}, which means that they satisfy
the relation~$\trans(n,\energy)^T \left(\begin{smallmatrix}
	0 & -1 \\ 1 & 0
  \end{smallmatrix}\right)
 \trans(n,\energy) 
  = \left(\begin{smallmatrix}
	0 & -1 \\ 1 & 0
  \end{smallmatrix}\right)$.
Especially, one knows that eigenvalues of a symplectic matrix always come in pairs $\{\lambda_i, \lambda_i^{-1} \}$ and, therefore, their determinant is $1$.
For a periodic potential with period $\period$, we define the \emph{monodromy matrix}
	\begin{equation}\label{eq:TransferMDefinition}
        M(\energy) \coloneqq  \trans({\period-1,\energy})\, \cdots \,\trans({1,\energy})\, \trans({0,\energy})
 	\,.
	\end{equation}
Obviously, $M(E)$ brings us from ${x_0\choose x_{-1}}$ to ${x_K\choose x_{K-1}}$, and then again from ${x_K\choose x_{K-1}}$ to ${x_{2K}\choose x_{2K-1}}$, and so on.
   When it is clear from the context that we only consider the energy~$E = 0$, we will write $T(n) \coloneqq T(n,0)$ and $M \coloneqq M(0)$.

As we have seen in Section~\ref{sec:invAndLimOps}, limit operators provide a useful tool for studying Fredholmness and invertibility of operators. 
The next lemma gives a complete description of the limit operators for periodic Schrödinger operators.

\begin{lemma}\label{lem:limOpsPeriodic}
  Assume Hypothesis~\ref{hypo:PeriodicH}. 
  \begin{enumerate}[label=(\alph*)]
    \item\label{it:limOpsPeriodic}
    Every limit operator of $H$ is again a $K$-periodic Schrödinger operator and
    \begin{equation*}
      \Lim(H) = \{ S^{-k} H S^k : k = 0,\dots,K - 1\} = \Lim_+(H) = \Lim_-(H) \,,
    \end{equation*}
    where $S$ denotes the right shift operator, i.e.\ $(S x)_n = x_{n - 1}$ for all $n \in \ZZ$.
  \item\label{it:limOpsPeriodicSpec} $\sigma(B) = \sigma(H)$ for all $B \in\Lim(H)$.
  \item\label{it:limOpsPeriodicTrace} If $B \in \Lim(H)$ with monodromy matrix~$M_B(E)$, then we have the identity $\tr(M_B(\energy)) = \tr(M(\energy))$ for all $\energy \in \RR$.
  \end{enumerate}
\end{lemma}

\begin{proof}
  \ref{it:limOpsPeriodic} Let $B \in \Lim(H)$ with corresponding sequence $(h_k)$.
  By Definition~\ref{def:limOps}, the sequence of representation matrices $(H_k)_{k \in \NN}$ with $H_k \coloneqq S^{-h_k} H S^{h_k}$ converges entrywise to the representation matrix of $B$.
  As the diagonal of $H$ only consists of the periodic continuation of $w \coloneqq (v(0),\dots,v(K - 1))$, the diagonal of the shifted matrix $H_k$ only consists of a periodic continuation of a cyclic permutation of $w$.
  Hence, in order to converge, the sequence $(H_k)$ has to become constant, eventually.
  This proves that $B = S^{-k} H S^k$ for some $k \in \{0,\dots,K  - 1\}$.

  The other inclusion is straightforward: choosing $h=(k+n\cdot K)_n$, we get
  \begin{equation*}
    S^{-k} H S^k = \lim_{n \to \infty} S^{-(k + n\cdot K)} HS^{k + n\cdot K} = A_h\in\Lim(H)\,.
  \end{equation*}

  \ref{it:limOpsPeriodicSpec} This follows from the fact that the shift operator is an isomorphism on $\ell^p(\ZZ)$.

  \ref{it:limOpsPeriodicTrace} From part~\ref{it:limOpsPeriodic} it follows that the monodromy matrix for $B$ is given by 
  \begin{equation*}
    M_B(E) = T(\tau(K-1), E) \cdots T(\tau(1),E) \,T(\tau(0),E) \, ,
  \end{equation*}
   where $(\tau(K-1),\dots,\tau(1),\tau(0))$ is just a cyclic permutation of $(0,1,\dots,K-1)$.
  Consequently $\tr(M_B(E)) = \tr(M(E))$.
\end{proof}

The complete description of the limit operators of a periodic Schrödinger operator from Lemma~\ref{lem:limOpsPeriodic} now allows for a powerful characterisation of invertibility.

\begin{lemma}\label{lemma:bijectiveIFFinjective_Marko1}
  Assume Hypothesis~\ref{hypo:PeriodicH}.
  The periodic Schrödinger operator $H$ is invertible on $\ell^p(\ZZ)$ for any $p\in [1,\infty]$ if and only if $H$ is injective on~$\ell^\infty(\ZZ)$.
\end{lemma}

\begin{proof}
  Let $H$ be injective on $\ell^\infty(\ZZ)$. 
  By Lemma~\ref{lem:limOpsPeriodic}\ref{it:limOpsPeriodic}, all limit operators of $H$
  are shifts of $H$ 
and hence, 
also injective on $\ell^\infty(\ZZ)$.
  In particular, by Lemma~\ref{lem:limOpsPeriodic}\ref{it:limOpsPeriodic} we have $H \in \Lim(H)$.
  Now, Corollary~\ref{cor:fundamental} implies that $H$ is invertible on $\ell^p(\ZZ)$ for all $p \in [1,\infty]$. 
  The only-if-part also follows from Corollary~\ref{cor:fundamental}.
\end{proof}

The following proposition resembles a well-known result about the spectrum of periodic Schrödinger operators, which can
be described by a trace condition of the monodromy matrix, see, e.g.,~\cite{Puelz.2014}.
We still give a proof here, because it employs limit operator techniques and is not very long.

\begin{proposition}\label{prop:trace_condition1}
  Assume Hypothesis~\ref{hypo:PeriodicH} and let $E \in \RR$.
  Then $E \in \sigma(H)$ if and only if the so-called \emph{trace condition},~$\abs{\tr(M(E))} \leq 2$, holds.
 \end{proposition}

\begin{proof}
  A real number~$\energy$ lies in the resolvent set if and only if $H-\energy$ is invertible on~$\ell^p(\ZZ)$. 
  By Lemma~\ref{lemma:bijectiveIFFinjective_Marko1}, this is equivalent to $H-\energy$ being injective on $\ell^\infty(\ZZ)$. 
  However, using~\eqref{eq:Jac_recursion_formula2}, this is equivalent to the claim that for all $(\alpha, \beta) \neq (0,0)$, the two-sided sequence
        \begin{equation*}
 	\Big(	M(\energy)^n \begin{pmatrix}
 			\alpha \\
 			\beta
 		\end{pmatrix}
 	\Big)_{n \in \ZZ}
        \end{equation*}
 is unbounded. This is satisfied if and only if both eigenvalues $\lambda_1$, $\lambda_2$ of $M(\energy)$ fulfil $\abs{\lambda_i} \neq 1$. 
 Since $M(\energy)$ has determinant $\det(M(\energy)) = 1$,
 this is equivalent to
 the characteristic polynomial $p(\lambda) = \lambda^2 - \tr(M(\energy)) \lambda + 1$ 
 having two distinct real solutions. Note here that two identical solutions would be $\pm 1$ and two proper complex solutions would be $\lambda_1$, $\overline{\lambda_1}$ with product $1$, hence $|\lambda_i|=1$.
However, having two distinct real solutions is equivalent to the discriminant of $p$ being positive, which yields the condition $\tr(M(\energy))^2 > 4$.
\end{proof}

In the following examples, we illustrate the use of the trace condition from Proposition~\ref{prop:trace_condition1} for determining the spectrum of periodic Schrödinger operators.

\begin{example}\phantomsection\label{ex:twosidedPeriodic}  
\begin{enumerate}[label=(\alph*)]
	\item If the potential $v$ is constant with $v(0) \in \RR$, which means $\period = 1$,
	then $\tr(M(\energy))=\energy-v(0)$ and
	$\sigma(H) = [-2 + v(0), 2 +v(0)]$. 
	
	\item If the potential $v$ has period $2$, with $v(0) = -1$ and $v(1) = 1$, 
	we get
	$\tr(M(E))=\energy^2-3$ and
	$\sigma(H) = [-\sqrt{5} , -1] \cup [1, \sqrt{5}]$. 
        \item\label{ex:twosidedPeriodic-2}
          More generally, for a $2$-periodic potential, 
          we have 
          \begin{equation*}
            \tr(M(\energy)) = -2 + (\energy - v(0))(\energy - v(1))\,.
          \end{equation*}
          This shows that we always will have two spectral bands, namely
          \begin{align*}
            \sigma(H) =  \Big[\frac{1}{2}(v(0) + v&(1) - \delta),\,   \min\{v(0), v(1)\} \Big] \\
            \cup &
            \Big[ \max\{v(0),v(1)\}  ,\,  \frac{1}{2}(v(0) + v(1) + \delta)      \Big]\, ,
          \end{align*}
          where $\delta = \sqrt{16 + v(0)^2 - 2 v(0) v(1) + v(1)^2}$ denotes the discriminant of the polynomial $\tr(M(\energy))$.
	\item\label{ex:twosidedPeriodicC} For a potential $v$ with period $3$ given by $v(0) = 0$, $v(1) = 1$
	and $v(2) = 0$, we get
	$\tr(M(\energy)) = {{\energy}^{3}}-{{\energy}^{2}}-3 \energy+1 $. 
        Therefore, the spectrum is given by 
	$\sigma(H) = [-\sqrt{3} , -1] \cup [1-\sqrt{2}, 1] \cup [1+\sqrt{2}, \sqrt{3}]$. 
\end{enumerate}
\end{example}

\begin{corollary}\label{cor:bandStructure}
    Assume Hypothesis~\ref{hypo:PeriodicH}.
The spectrum of~$H$ consists of at most~$K$ bands. 
More precisely, there are
numbers~$a_1 \leq b_1 \leq a_2 \leq \cdots \leq a_K \leq b_K$
with
\begin{equation*}
  \sigma(H) = \bigcup_{i =1}^K [a_i, b_i] \,.
\end{equation*}
\end{corollary}

\begin{proof}
 The trace of $M(E)$ gives us a polynomial $p_K$ in the variable $E$ with degree $K$.
 A real number $E$ lies in $\sigma(H)$
 if and only if $-2 \leq p_K(E) \leq 2$, by Proposition~\ref{prop:trace_condition1}. 
 The real zeros of the polynomials $p_K \pm 2$ give us the numbers~$\{a_i, b_i\}_{i =1, \ldots, K}$.
\end{proof}

\subsection{One-Sided Periodic Schrödinger Operators}\label{sec:oneSidedPeriodic}
In this section, we will cut the Schrödinger operator in
two different ways to get one-sided operators.

If $H$ is a periodic Schrödinger operator on~$\ell^p(\ZZ)$,
we let $H_+$ denote its one-sided restriction to the space~$\ell^p(\ZZ_+)$ subject to the Dirichlet boundary condition~$x_{-1} = 0$.
Analogously, let~$H_-$ denote its one-sided restriction to the space~$\ell^p(\ZZ_-)$ subject to the Dirichlet boundary condition~$x_{1} = 0$.
The matrix representation of~$H_\pm$ coincides with the respective one-sided compressions from Definition~\ref{def:limOps}.

The following lemma allows us to restrict ourselves to 
one of the two compressions, $H_+$ or $H_-$.

\begin{lemma}\label{lem:specH-}
  Assume Hypothesis~\ref{hypo:PeriodicH}. Let
  $H^{\mathrm{R}}$ denote the Schrödinger operator with the \emph{reversed potential} $v^{\mathrm{R}}(n) \coloneqq v(-n)$, $n \in \ZZ$.
  Then $\sigma(H) = \sigma(H^{\mathrm{R}})$, $\sigma(H_-) = \sigma(H^{\mathrm{R}}_+)$, and $\tr(M_H) = \tr(M_{H^{\mathrm{R}}})$.
\end{lemma}
\begin{proof}
  Consider the flip operators
  \begin{equation*}
    \Phi \colon \ell^p(\ZZ) \to \ell^p(\ZZ), \quad (x_n)_{n \in \ZZ} \mapsto (x_{-n})_{n \in \ZZ} 
  \end{equation*}
  and
  \begin{equation*}
    \Phi_- \colon \ell^p(\ZZ_+) \to \ell^p(\ZZ_-), \quad (x_n)_{n \in \ZZ_+} \mapsto (x_{-n})_{n \in \ZZ_-} \, .
  \end{equation*}
  Clearly, $\Phi$ and $\Phi_-$ are isomorphisms. The claims for the spectra follow from the identities
  $H = \Phi H^{\mathrm{R}} \Phi^{-1}$ 
  and 
  \begin{equation}\label{eq:H_flipHplus}
   H_- = \Phi_- H^{\mathrm{R}}_+ \Phi_-^{-1} \,,
  \end{equation}
	which is a straightforward calculation.
	
	To prove the trace identity, we use that a transfer matrix $T = T(n,E)$ always has the property:
	\begin{equation*}
	T^{-1} = F T F \quad \text{with} \quad F = \begin{pmatrix}
	0 & 1 \\ 1 & 0
\end{pmatrix}	\,. 
	\end{equation*}
	Hence, we also have $F M_H F = T(K-1)^{-1} \cdots T(0)^{-1} = M_{H^\mathrm{R}}^{-1}$. Taking the trace yields:
	\begin{equation*}
	 \tr(M_H) = \tr( F M_H F ) = \tr \big(M_{H^{\mathrm{R}}}^{-1} \big) = \tr \big(M_{H^{\mathrm{R}}} \big) \,,
	\end{equation*}
	where the last equality holds for every real symplectic matrix.
\end{proof}

Since $H\in\Lim(H_+)$, the spectrum of $H_+$ contains that of $H$ as a consequence of Proposition~\ref{prop:essSpecEqualsSpecnew}\ref{it:spectra_of_limops}. But how much bigger is it? The answer clearly has something to do with the entries of $H_+$ near the cut-off. To this end, put
\begin{equation*}
 H_{a..b} \coloneqq  
		\begin{pmatrix}
			v(a)       & 1  &        &        &           \\
			1     & v(a+1)    &  \ddots  &        &               \\
			&         \ddots & \ddots & \ddots  &        \\
			&                & \ddots & v(b-1) & 1  \\
			&                &        & 1  & v(b)
		\end{pmatrix}
\end{equation*}
for $a, b\in\ZZ_+$ with $a\le b$.
Then indeed, the following proposition allows us to step from the two-sided to the one-sided periodic operator in terms of spectra. It follows directly from~\cite[Theorem~4.42(i)]{Hagger.2016}, where it is even stated for non-self-adjoint operators with three periodic diagonals. 
\begin{proposition}\label{prop:oneSidedSpecRep}
  Assume Hypothesis~\ref{hypo:PeriodicH} 
  with period $\period \geq 2$. 
    Then 
    \begin{equation}\label{eq:oneSidedSpecRep}
        \sigma(H_+) = \sigma(H) \cup \big\{ E \in \mathbb{R} :
        M(E)_{2,1} = 0 ~\text{ and }~
         | M(E)_{1,1} | < 1 
         \big\} \,.
    \end{equation}
    For the entries of $M(E)$ we have the formula
    \begin{equation}\label{eq:monodromyDeterminant}
    M(E)
    = (-1)^{K + 1}\begin{pmatrix}
            -\det(H_{0..K-1}-E) & -\det(H_{1..K-1}-E) \\[1em]
      \hphantom{-}\det(H_{0..K-2}-E) & \hphantom{-}\det(H_{1..K-2}-E )
    \end{pmatrix} \,.
   \end{equation}
\end{proposition}

\begin{proof}
  Note that $\sigma(H) = \sigma_\ess(H_+)$ by Lemma~\ref{lem:limOpsPeriodic}\ref{it:limOpsPeriodicSpec}, Proposition~\ref{prop:essSpecEqualsSpecnew}\ref{it:equal_spectra_ess}, and Lemma~\ref{lem:limOpsPeriodic}\ref{it:limOpsPeriodic}.
    Since $H$ is self-adjoint on $\ell^2(\ZZ)$, we know that $\sigma(H)$,
    which is the same for $p\ne 2$,
    lies on the real axis.
    This implies that~$\CC \setminus \sigma_\ess(H_+)$ is connected and thus unbounded.
    With this, identity~\eqref{eq:oneSidedSpecRep} is a direct consequence of~\cite[Theorem~4.42(i)]{Hagger.2016} as the tridiagonal matrix representation of $H_+$ has constant upper and lower diagonals. 
    The formula for $M(E)$ is from~\cite[Section~4.2.3]{Hagger.2016}.
\end{proof}

Note that we have used Proposition~\ref{prop:essSpecEqualsSpecnew}, which is a general result about the essential spectrum
of band operators using its limit operators. 
However, in the case of periodic Jacobi operators,
also other proofs for the essential spectrum $\sigma_\ess(H_+)$
are known, see, e.g.,~\cite[Theorem~7.2.1]{Simon.2011}.

Now, we present an interlacing property that holds in general
for the eigenvalues of Jacobi matrices. 
Here, these matrices come from a $\period$-periodic Schrödinger operator $H$.
Then we can consider the following matrices:
\begin{equation*}
H_{0..\period-2}\qquad\text{and}\qquad
H_{0..\period-1}^\varphi \coloneqq H_{0..\period-1} + 
\begin{pmatrix}  & & & & \mathrm{e}^{-\ii \varphi}\\   \\   \mathrm{e}^{\ii \varphi} \end{pmatrix}
\end{equation*}
for any $\varphi \in [0,2 \pi]$.
Both correspond to finite-dimensional restrictions of the operator equation with Dirichlet and periodic boundary conditions, respectively.
The eigenvalues of 
$H_{0..\period-1}^\varphi$ for all values of $\varphi$
give us the spectrum of the (two-sided) periodic Schrödinger operator in the following sense
\begin{equation*}
	\sigma(H) = \bigcup_{\varphi \in [0,2\pi]} 
	   \sigma(H_{0..\period-1}^\varphi) =
	\bigcup_{j =1}^{\period} [\energy_{2j-1}, \energy_{2j}]\,,
\end{equation*}
where $\energy_{1},\dots,\energy_{2\period}$ is the increasing enumeration of the eigenvalues of 
$H_{0..\period-1}^0$ and $H_{0..\period-1}^\pi$
together.
Again, one gets the band structure of the periodic Schrödinger operator as shown in Corollary~\ref{cor:bandStructure}.
For more details about this approach, the so-called \emph{Floquet--Bloch theory} for periodic band operators, 
see, e.g.,~\cite[Thm 4.4.9]{Davies2007:Book},~\cite{Puelz.2014},~and~\cite[Chapter~7]{Teschl.2000}. In addition, one finds
the following classical interlacing result, cf.~\cite[Theorem~4.5]{Teschl.2000}.

\begin{proposition}\label{prop:SturmOnedim-interlacing1}
 Assume Hypothesis~\ref{hypo:PeriodicH}.
 The eigenvalues of 
 $H_{0..\period-2}$ and $H_{0..\period-1}^\varphi$
 are interlacing for every
 $\varphi \in [0,2 \pi]$, that is
		\begin{equation*}
		\widetilde{\energy}_{1,\varphi} \leq \energy_1 \leq \widetilde{\energy}_{2,\varphi}  \leq \energy_2 \leq \cdots
		 \leq \energy_{\period-1} \leq \widetilde{\energy}_{\period,\varphi}
		\end{equation*}
where $\energy_1 \le \dots \le \energy_{\period-1}$ and $\widetilde{\energy}_{1,\varphi} \le \dots \le \widetilde{\energy}_{\period,\varphi}$ denote the eigenvalues of  
$H_{0..\period-2}$ and $H_{0..\period-1}^\varphi$, respectively.
\end{proposition}

\begin{proof}
The proof can be done analogously to~\cite[Proposition 2.2]{Simon.2005sturm}
by using the \emph{min-max principle}, see, \eg,~\cite[Theorem~4.2.11]{Horn.1985}~and~\cite[Theorem~12.1]{Schmudgen.2012}.
 \end{proof}
 Applying this result, we can supplement Proposition~\ref{prop:oneSidedSpecRep} with more information. There we 
 have seen that the spectrum of the restricted operator $H_+$ can only be larger
 than the spectrum of $H$ and the difference can only be given by eigenvalues. 
 These eigenvalues of $H_+$ coincide with those of $H_{0..\period-2}$. 
 However,
 now we conclude with Proposition~\ref{prop:SturmOnedim-interlacing1}
 that these eigenvalues necessarily lie inside the gaps of $\sigma(H)$
 and that each gap can only add at most one eigenvalue to $\sigma(H_+)$.
 Since they stem from a cut-off, i.e., imposing a Dirichlet condition, they are often referred to as \emph{Dirichlet eigenvalues}.
 
\begin{example}\label{ex:gapsPeriodic}
\begin{enumerate}[label=(\alph*)]
  \item\label{it:gapsPeriodic-2} Consider the $2$-periodic Schrödinger operator $H$ from Example~\ref{ex:twosidedPeriodic}\ref{ex:twosidedPeriodic-2}. 
    For the spectrum of $H_+$, we calculate, according to Proposition~\ref{prop:oneSidedSpecRep},
        \begin{equation*}
      \sigma(H_{0..2-2}) = \sigma\begin{pmatrix}v(0)\end{pmatrix} = \{v(0)\}
    \end{equation*}
    and
    \begin{equation*}
      \det(H_{0..1} - v(0))
      = 
      \det\Big(
    \begin{pmatrix}
    v(0) - v(0) & 1 \\ 
    1 & v(1) - v(0)
    \end{pmatrix}\Big) 
    = -1,
    \end{equation*}
    which implies $\sigma(H_+) = \sigma(H)$.
\item For the 3-periodic Schrödinger operator from Example~\ref{ex:twosidedPeriodic}\ref{ex:twosidedPeriodicC}, we find:
    \begin{equation*}
  \sigma(H_+) = \sigma(H) \cup \Big\{ -\frac{\sqrt{5}-1}{2} \Big\}\,.
    \end{equation*}
\end{enumerate}
\end{example}

As we have seen, Proposition~\ref{prop:trace_condition1} with $E = 0$ gives us a sufficient condition for a periodic Schrödinger operator to be invertible. 
In order to ensure applicability of the FSM, according to Proposition~\ref{prop:FSMCondition}, we also need to ensure injectivity of one-sided compressions of periodic Schrödinger operators.
Note that, for a one-sided Schrödinger operator, the transfer matrices~$T(n)$ still follow the same pattern, as the Dirichlet boundary condition is encoded in the vector~$(x_n)$ during the start of the recursion from formula~\eqref{eq:Jac_recursion_formula2}.
Also recall that $M \coloneqq M(0)$ abbreviates the monodromy matrix for energy $E=0$.

\begin{proposition}\label{rem:Monodromy_Hplus_not_inject}
Assume Hypothesis~\ref{hypo:PeriodicH}.
If $H$ is invertible, that is, $|\tr(M)| > 2$, the following are equivalent:
 \begin{enumerate}[label=(\roman*)]
 \item\label{it:compressionInj} the compression $H_+$ is not injective on $\ell^\infty(\ZZ_+)$,
 \item\label{it:compressionInv} the compression $H_+$ is not invertible on any $\ell^p(\ZZ_+)$,
 \item\label{it:eigenvector} ${1\choose 0}$ is an eigenvector of $M$ w.r.t.\ the smaller (in modulus) of the two eigenvalues (the one with $|\lambda_i|<1$),
 \item\label{it:matrixEntries} $M_{2,1} = 0$ and $\abs{M_{1,1}} < 1$.
 \end{enumerate}
\end{proposition}

\begin{proof}
\ref{it:compressionInv} is independent of $p\in[1,\infty]$ by Proposition~\ref{prop:kurbatov}.
\ref{it:compressionInj}$\Rightarrow$\ref{it:compressionInv} is trivial for $p=\infty$, and \ref{it:compressionInv}$\Rightarrow$\ref{it:compressionInj} follows from Corollary~\ref{cor:injImpliesInvertibleForLimOps} for 
$p=2$. 

\ref{it:eigenvector}$\Leftrightarrow$\ref{it:matrixEntries} is obvious: the $2\times 2$-matrix $M$ has the eigenvector ${1\choose 0}$ if and only if $M_{2,1}=0$. The corresponding eigenvalue is then $M_{1,1}$.

It remains to show the equivalence of \ref{it:compressionInj} and \ref{it:eigenvector}.
To this end, assume that \ref{it:compressionInj} holds. This means that $H_+x=0$ has a bounded solution $x\in\ell^\infty(\ZZ_+)\setminus\{0\}$.
We reflect the homogeneous Dirichlet condition by putting $x_{-1} = 0$, and we let $x_0 \ne 0$ since otherwise the recursion~\eqref{eq:Jac_recursion_formula2} leads to $x=0$. By scaling, we can put $x_0=1$.  
The recursion formula~\eqref{eq:Jac_recursion_formula2} for the one-sided Schrödinger operator then gives us
    \begin{equation*}
         \begin{pmatrix}
            x_K\\
            x_{K - 1}
        \end{pmatrix}
        = 
        M \begin{pmatrix} x_0 \\ x_{-1}\end{pmatrix}
        = 
        M \begin{pmatrix} 1 \\ 0\end{pmatrix}
    \end{equation*}
 and furthermore
    \begin{equation}\label{eq:M^n10}
        \begin{pmatrix}
            x_{nK}\\
            x_{n K - 1}
        \end{pmatrix}
        = M^n \begin{pmatrix} 1 \\ 0 \end{pmatrix}
    \end{equation}
    for all $n \in \ZZ_+$.
    
    By assumption, the condition $|\tr(M)| > 2$ holds. 
    As in the proof of Proposition~\ref{prop:trace_condition1}, this gives for the eigenvalues of $M$, denoted by $\lambda_1, \lambda_2$, the chain of inequalities $|\lambda_1| > 1 > |\lambda_2|$.
    Let furthermore $(\xi_1, \xi_2)$ be a corresponding basis of eigenvectors of $M$.
    Because $x$ is bounded, also the vector sequence
    \begin{equation*}
        \left(M^n \begin{pmatrix} 1 \\ 0 \end{pmatrix}\right)_{n \in \ZZ_+}
    \end{equation*}
    is bounded.
    Describing $(1,0)^T$ with respect to the basis of eigenvectors yields
    \begin{equation*}
        M^n \begin{pmatrix} 1 \\ 0 \end{pmatrix} = M^n (  \alpha \xi_1 + \beta \xi_2   ) = \alpha \lambda_1^n \xi_1 + \beta \lambda_2^n \xi_2\,,\quad n\in\ZZ_+.
    \end{equation*}
    The boundedness of this sequence, together with  $|\lambda_1| > 1 > |\lambda_2|$,
    implies that $\alpha= 0$, so that ${1\choose 0}$ does not have a component into the direction $\xi_1$, i.e.\ it is an eigenvector of $M$ for the eigenvalue $\lambda_2$. So we arrive at \ref{it:eigenvector}.
    
    Finally, from \ref{it:eigenvector} we conclude, by the same arguments, the boundedness of the vector sequence~\eqref{eq:M^n10}, say $\|{x_{nK}\choose x_{nK-1}}\|_\infty<C<\infty$ for all $n\in\ZZ_+$. Since
    \[
    \begin{pmatrix}
            x_{nK+j}\\
            x_{n K - 1+j}
        \end{pmatrix} \ =\ T(j-1)\cdots T(1)T(0)
        \begin{pmatrix}
            x_{nK}\\
            x_{n K - 1}
        \end{pmatrix}
    \]
    for all $j\in\{1,\dots,K\}$, we get that 
\begin{equation*}    
|x_i| \leq 
\left\|
\begin{pmatrix}
  x_{i}\\
  x_{i-1}
\end{pmatrix}
\right\|_\infty
  \leq \max_{j=1, \ldots,K} \|T(j-1)\cdots T(1)T(0)\|C\,,\quad i\in\ZZ_+\,.
\end{equation*}
Hence, the solution $x$ of the equation $H_+x=0$ is bounded, which shows \ref{it:compressionInj}.
\end{proof}

Note that  $\abs{M_{1,1}} = 1$ is impossible in Proposition~\ref{rem:Monodromy_Hplus_not_inject}\ref{it:matrixEntries} if $M_{2,1} = 0$ because,
by $\det(M)=1$, then both eigenvalues of $M$ would have to have modulus one, 
contradicting $|\tr(M)|>2$.
Also note that the equivalence of \ref{it:compressionInv} and \ref{it:matrixEntries} can also be seen directly by Proposition~\ref{prop:oneSidedSpecRep}. 
We will mostly use the equivalence of their negations:
\begin{equation}\label{eq:H+inv}
H_+ \text{ is invertible on all }\ell^p(\ZZ_+)\quad\iff\quad M_{2,1}\ne 0\text{ or } |M_{1,1}|>1\,.
\end{equation}

Proposition~\ref{rem:Monodromy_Hplus_not_inject} not only yields the criterion~\eqref{eq:H+inv}, leading to our efficient test~\eqref{eq:checkM}, it is also the key to the following result about integer-valued potentials.

\begin{proposition}\label{prop:oneSidedPeriodic}
    Assume Hypothesis~\ref{hypo:PeriodicH} 
    with $p = \infty$ and $v(n) \in \mathbb{Z}$ for all $n \in \ZZ$. 
    If~$|\tr(M)| > 2$, that means, if $H$ is invertible, then
    \begin{enumerate}[label=(\alph*)]
      \item\label{it:H+inj}   $H_+$ is injective on $\ell^\infty(\ZZ_+)$ and 
      \item\label{it:H-inj}   $H_-$ is injective on $\ell^\infty(\ZZ_-)$\,.
\end{enumerate}
\end{proposition}

\begin{proof}
\ref{it:H+inj}
 Let $H$ be invertible 
 and assume that our claim, injectivity of $H_+$ on $\ell^\infty(\ZZ_+)$, was false.
 Then, by Proposition~\ref{rem:Monodromy_Hplus_not_inject}\ref{it:compressionInj}$\Leftrightarrow$\ref{it:matrixEntries}, $M_{2,1} = 0$ and $\abs{M_{1,1}} < 1$.
 In particular, $M$ is upper triangular and $M_{1,1}$ and $M_{2,2}$ are its eigenvalues.
 However, $v(n)\in\ZZ$ for all $n\in\ZZ$ implies $M\in\ZZ^{2\times 2}$ and hence, both 
 $M_{1,1}$ and $M_{2,2}$ are integer. Together with $\abs{M_{1,1}} < 1$ this implies that the eigenvalue $M_{1,1}$ is zero, which is impossible since $\det(M)=1$.
 
    \ref{it:H-inj} 
    Since $H$ is invertible, $H^{\mathrm{R}}$ is invertible by Lemma~\ref{lem:specH-}.
    Applying \ref{it:H+inj} to $H^{\mathrm{R}}$ in place of $H$ shows that $H_+^{\mathrm{R}}$ is injective on $\ell^\infty(\ZZ_+)$.
    But by
    \eqref{eq:H_flipHplus} from the proof of Lemma~\ref{lem:specH-}, $H_-$ is injective, too.
\end{proof}

\subsection{Periodic Schrödinger Operators with \texorpdfstring{$\{0,\lambda\}$}{\{0, lambda\}}-Valued Potentials}\label{sec:0,lambda}
Before we prove our main theorems, we take a closer look at periodic Schrödinger operators with $\{0,1\}$-valued potentials and their relatives.
To balance between the Laplace interaction term and the pointwise multiplication by the potential, one often introduces a so-called \emph{coupling constant} $\lambda>0$ as a weight for the $\{0,1\}$-valued potential $v$. The result is a discrete Schrödinger operator with a $\{0,\lambda\}$-valued potential. In the examples announced in Remark~\ref{rem:12} and Remark~\ref{rem:14}, the potential is of this kind.

For $K\in\NN$ and $w\in\{0,1\}^K$, let $v\in\{0,\lambda\}^\ZZ$ be the periodic extension of $\lambda\cdot w$. We consider the discrete Schrödinger operator $H$ with potential $v$. Note that we do not specify $\lambda$ yet but keep it as a variable during most of the following computations.

For applicability of the FSM to $H$, we are, by Proposition~\ref{prop:FSMCondition}, interested in the invertibility, and hence the spectra, of
\begin{enumerate}[label=(\alph*)]
	\item $H$,
	\item all $L_+$ with $L\in\Lim_-(H)$,
	\item all $R_-$ with $R\in\Lim_+(H)$.
\end{enumerate}
So let us look at the union of spectra
\begin{equation}\label{eq:seaweed-set}
\sigma(H)\quad\cup\bigcup_{B\in\Lim(H)}\big(\sigma(B_+)\cup\sigma(B_-)\big)\,.
\end{equation}
We know from Section~\ref{sec:oneSidedPeriodic} that $\sigma(H)$ consists of at most $K$ closed spectral bands. 
Further, for $B\in\Lim(H)$, the spectra $\sigma(B_+)$ and $\sigma(B_-)$ are supersets of $\sigma(B)=\sigma(H)$, only larger by at most one Dirichlet eigenvalue per spectral gap. 
Let us plot the set~\eqref{eq:seaweed-set}, while varying the coupling constant $\lambda>0$.
Figure~\ref{fig:K4} shows that plot for $w=(1,1,0,1)$.  
\begin{figure}[htbp]
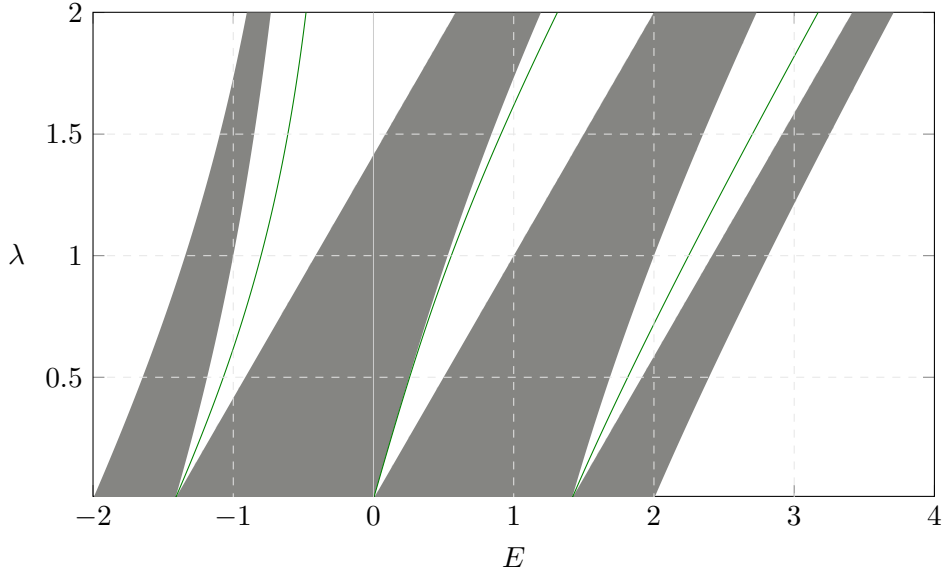
 


 	\caption{ \footnotesize
		Union~\eqref{eq:seaweed-set} of spectra for $H$ with potential $v=\lambda\cdot(1,1,0,1)$, periodically extended, while $\lambda$ changes along the vertical axis. The spectral bands are shown in grey and the Dirichlet eigenvalues in green. If we look at the vertical line at energy $E=0$, the plot suggests that, whenever $H$ is invertible (not grey), then all the $B_+$ and $B_-$ are invertible (not green). In other words: no green line is crossing the vertical line at $E=0$.
Hence, $H$ should be FSM-simple, where the rigorous proof is given in Section~\ref{sec:0,lambda-systematic}.
	}\label{fig:K4}
\end{figure}		
Figure~\ref{fig:K5} shows the union~\eqref{eq:seaweed-set} for $w=(1,1,0,1,0)$, and Figure~\ref{fig:K9} for $w=(1,1,0,1,0,1,0,1,1)$.
\begin{figure}[htbp]
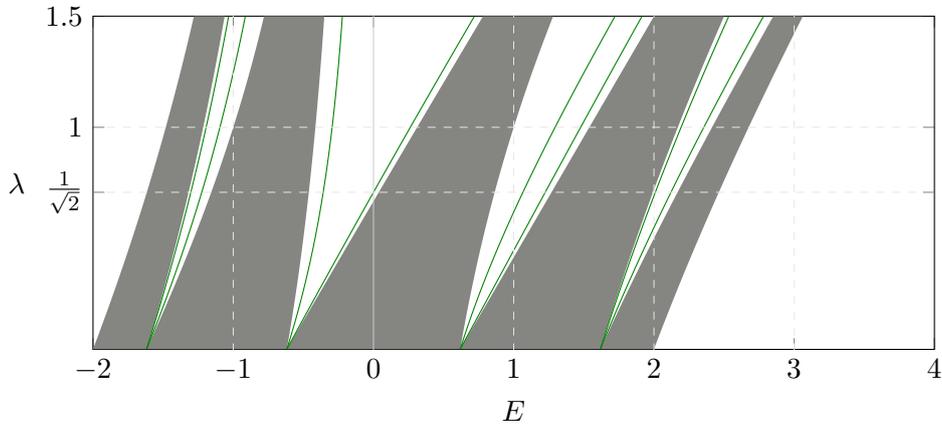



 	\caption{  \footnotesize
            Union~\eqref{eq:seaweed-set} of spectra for $H$ with potential $v=\lambda\cdot(1,1,0,1,0)$, periodically extended, while $\lambda$ changes along the vertical axis. The spectral bands are shown in grey and the Dirichlet eigenvalues in green. This time we see that one Dirichlet eigenvalue crosses the vertical line at energy $E=0$ and height $\lambda=\frac 1{\sqrt 2}$. So, for this $\lambda$, $H$ is invertible (not grey), but the FSM is not applicable (green). In Section~\ref{sec:0,lambda-systematic}, we detect this example algebraically, and in Example~\ref{ex:5per} it is proven that this operator is indeed not FSM-simple.
	}\label{fig:K5}
\end{figure} 
 \begin{figure}[htbp]
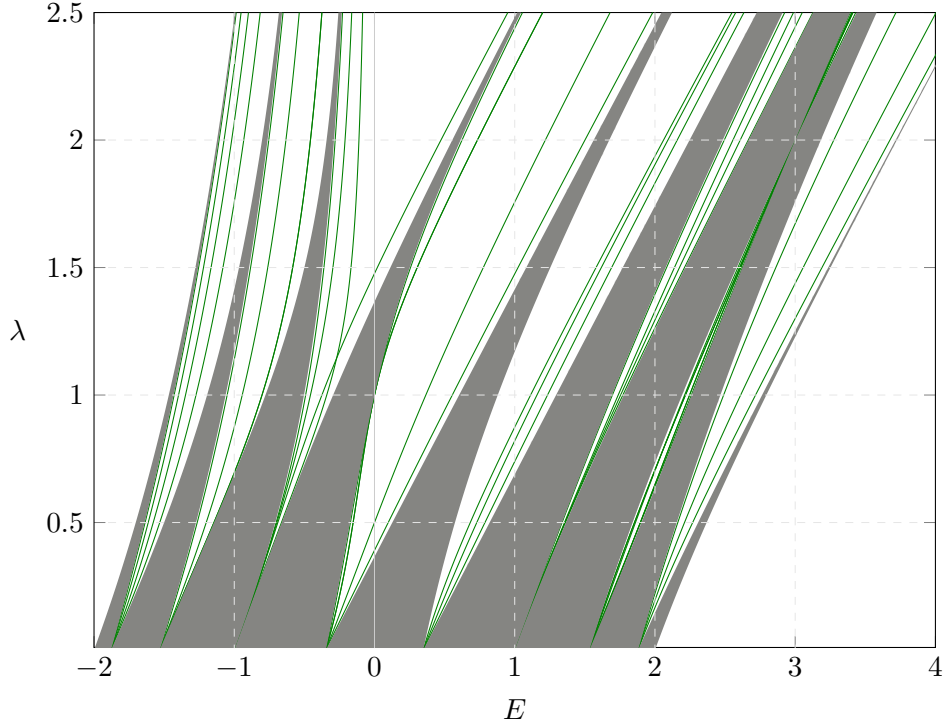
 


 	\caption{  \footnotesize
            Union~\eqref{eq:seaweed-set} of spectra for $H$ with the 9-periodic potential $v=\lambda\cdot(1,1,0,1,0,1,0,1,1)$, 
            periodically extended, while $\lambda$ changes along the vertical axis. The spectral bands are shown in grey and the Dirichlet eigenvalues in green. 
            At~$E = 3$ and $\lambda = 2$, two spectral bands merge into one before breaking up again.       
                   This time we see that one Dirichlet eigenvalue crosses the vertical line at energy $E=0$ and height $\lambda=\frac{1}{2}$. So, for $\lambda=\frac{1}{2}$, $H$ is invertible (not grey) but the FSM is not applicable (green). Unlike for periods $K<9$, this crossing happens at a rational value of $\lambda$.       
                   In Section~\ref{sec:0,lambda-systematic}, we show how to detect these examples algebraically, and in Example~\ref{ex:9per} it is proven that this operator with $\lambda=\frac{1}{2}$ is indeed not FSM-simple.		
	}\label{fig:K9}
\end{figure}
The latter two choices of $w$ will be revisited in Section~\ref{sec:examples}.

\section{Proofs of Our Main Theorems}\label{sec:proofs}
In this final section, we bring our results on the invertibility of one- and two-sided periodic Schrödinger operators together to conclude, via Proposition~\ref{prop:FSMCondition}, our Theorems~\ref{thm:M21_M11} and \ref{thm:integerFSM} about the FSM. 
In fact, it is of advantage to give the two proofs in this order -- as the first and last subsection of Section~\ref{sec:proofs}. 
In between, we insert a subsection illustrating the use of Theorem~\ref{thm:M21_M11} with examples and a subsection that deals with an algorithm for their systematic study of $\{0,\lambda\}$-valued potentials. 
The latter contributes to the proof of Theorem~\ref{thm:integerFSM} that, as already mentioned, forms the last subsection.

\begin{samepage}
The proofs of this section  will exploit the following variant of Proposition~\ref{prop:FSMCondition} which is made possible by Lemma~\ref{lem:specH-} and relies only on one-sided compressions on $\ZZ_+$ of the limit operators of $H$ and $H^{\mathrm{R}}$:
\begin{align} 
\begin{array}{c}\text{FSM applicable to $H$}\end{array}
\begin{tabular}{cc}\text{\tiny Prop.~\ref{prop:FSMCondition} + Lem.~\ref{lem:specH-}}\\ $\iff$ \\\  \end{tabular}
\left\{
\begin{array}{l} 
  \text{the following are invertible:}\\[0.5em]
  \text{(a)}\ \ H\\[0.5em]
  \text{(b)}\ \ \text{all }L_+\text{ with } L\in\Lim_-(H)\\[0.5em]
  \text{(c)}\ \ \text{all }\widetilde{L}_+\text{ with } \widetilde{L}\in\Lim_-(H^\mathrm{R})
\end{array}
\right.
\end{align}
\end{samepage}

\subsection{Proof of Theorem~\ref{thm:M21_M11}}

We start with statement~\ref{it:thmTwoSided}.
Recall that invertibility of $H$ is characterised by the trace formula thanks to Proposition~\ref{prop:trace_condition1}.
Consequently, we assume $H$ to be invertible for the rest of the proof. 

In order to complete the proof of \ref{it:thmTwoSided}, we want to use Proposition~\ref{rem:Monodromy_Hplus_not_inject} in order to conclude that the invertibility of the one-sided compressions $L_+$ and $\widetilde{L}_+$ is characterised by the condition
\begin{equation}\label{eq:checkM2}
M_{2,1}\ne 0\qquad\text{or}\qquad |M_{1,1}|>1\,,
\end{equation}
for their respective monodromy matrices. 

Indeed, by Lemma~\ref{lem:limOpsPeriodic}\ref{it:limOpsPeriodic}, the limit operators $L$ and $\widetilde{L}$ are again periodic Schrödinger operators, and as such they are invertible by Lemma~\ref{lem:fundamental} as a consequence of the invertibility of $H$ and $H^\mathrm{R}$.
As a result, Proposition~\ref{rem:Monodromy_Hplus_not_inject} applies and all compressions $L_+$ and $\widetilde{L}_+$ are invertible if and only if their monodromy matrices $M_{\vphantom{\widetilde{L}}L}$ and $M_{\widetilde{L}}$ are subject to Condition~\eqref{eq:checkM2}.
Note that due to Lemma~\ref{lem:limOpsPeriodic}\ref{it:limOpsPeriodic}, we have a one-to-one correspondence between the monodromy matrices $M_L$ for $L \in \Lim_-(H)$ and the matrices $M^{(j)}$ in the formulation of Theorem~\ref{thm:M21_M11}. More precisely, $L = S^{-j} H S^{j}$ for some $j \in \{0,\dots,K-1\}$ which yields $M_L = M^{(j)}$. 
A similar one-to-one correspondence also holds for the matrices $M_{\widetilde{L}}$ and $\widetilde{M}^{(j)}$ with  $j  \in \{ 0,\dots,K-1 \}$; this time $\widetilde{M}^{(j)}=M_{\widetilde{L}}$ with limit operator $\widetilde{L}=S^{j-1}H^{\mathrm{R}}S^{-j+1} \in \Lim_-(H^{\mathrm{R}})$.

For statement~\ref{it:thmOneSided}, we employ the second part of Proposition~\ref{prop:FSMCondition} which characterises applicability of the FSM to $H_+$ via invertibility of $H_+$ and of its limit operators $\widetilde{L}_+$ for $\widetilde{L} \in \Lim_-(H^{\mathrm{R}})$. 
As in the proof of part~\ref{it:thmTwoSided}, the invertibility of $H_+$ and the limit operators $\widetilde{L}_+$ is characterised by the Condition~\eqref{eq:checkM2} through Proposition~\ref{rem:Monodromy_Hplus_not_inject}.

Finally, we assume that $H$ is FSM-simple and that $H_+$ is invertible.  
Then $H$ is invertible by Proposition~\ref{prop:oneSidedSpecRep} and therefore the FSM is applicable to $H$, which implies that the FSM is applicable to $H_+$ as shown above.
Hence, $H_+$ is FSM-simple.
\qed

\begin{remark}
Note that the argument in the last part of the proof shows that the implication
\begin{equation*}
H\text{ is FSM-simple}\quad\Rightarrow\quad H_+\text{ is FSM-simple}
\end{equation*}
generalises to all band operators $A$ with the property $A\in\Lim_+(A)$. 
Indeed, if $A$ is FSM-simple then invertibility of $A_+$ implies, by $A\in\Lim_+(A)=\Lim(A_+)$ and Lemma~\ref{lem:fundamental}\ref{it:AFredholm}$\Rightarrow$\ref{it:allLimOpsInvertible}, that $A$ is invertible. Since $A$ is FSM-simple, the FSM is applicable to $A$ and, by Proposition~\ref{prop:FSMCondition}, it is applicable to $A_+$.
\end{remark}

\subsection{Proof of Remark~\ref{rem:12} and \ref{rem:14}}\label{sec:examples}
\begin{example}\label{ex:3per}
  The $3$-periodic Schrödinger operator $H$ with continuously repeated potential $v(0) = 2$, $v(1) = \frac{1}{2}$, and $v(2) = \frac{1}{2}$  has monodromy matrices
\begin{equation*}
  M^{(0)}=\begin{pmatrix}2&\frac 34\\[0.25em]0&\frac 12\end{pmatrix},\quad
  M^{(1)}=\begin{pmatrix}\vphantom{\frac 12}2&0\\[0.25em]-\frac 34&\frac 12\end{pmatrix},\quad
  M^{(2)}=\begin{pmatrix}\frac 12&0\\[0.25em] 0&\vphantom{\frac 12}2\end{pmatrix}.
\end{equation*}
The trace of $M^{(0)}$ is $2+\frac 12=\frac 52>2$, whence $H$ is invertible.
The traces of $M^{(1)}$ and $M^{(2)}$ are, not surprisingly, also $\frac 52$.
But now comes the problem:
\begin{equation*}
M^{(2)}_{2,1}=0\qquad\text{and}\qquad |M^{(2)}_{1,1}|=\frac 12<1
\end{equation*}
so that $M^{(2)}$ fails the test~\eqref{eq:checkM}, whence one particular $L_+$ with $L\in\Lim_-(H)$ is not invertible, by Proposition~\ref{rem:Monodromy_Hplus_not_inject}. More precisely,
  $L_+=(S^{-2}HS^2)_+ = (H_l)_+$ with $l = (l_n) =(-3n+2)$.

It is easy to check that $\widetilde{M}^{(2)}=M^{(2)}$
so that also $\widetilde{M}^{(2)}$ fails \eqref{eq:checkM}. 
This means that $R_- =  H_- = (H_r)_-$ with $r = (r_n)=(3n+1)$ is not invertible.

By Theorem~\ref{thm:M21_M11}\ref{it:thmTwoSided}, the FSM is not applicable to $H$ -- although $H$ is invertible.
Hence, $H$ is not FSM-simple, as already announced in Remark~\ref{rem:12}\ref{it:periodThree}.

One further checks that the matrices $\widetilde{M}^{(0)}=M^{(1)}$ and $\widetilde{M}^{(1)}=M^{(0)}$ pass the test~\eqref{eq:checkM} which allows to conclude that the applicability of the FSM only fails because of the two ``bad'' compressions of the limit operators $H_l$ and $H_r$ identified above.
  However, with the detailed knowledge about the underlying ``bad'' sequences $(l_n)=(-3n+2)$ and $(r_n)=(3n+1)$ corresponding to these limit operators, one can adapt the FSM, see, e.g.,~\cite{Lindner.2010,Lindner.2018,Rabinovich.2008}. In particular, one can show that the \emph{adapted FSM} for $H$ with $A_n=(a_{ij})_{i,j=l_n'}^{r_n'}$ is applicable if the cut-off sequences $(l_n')$ and $(r_n')$ asymptotically avoid the sequences $(l_n)=(-3n+2)$ and $(r_n)=(3n+1)$, respectively.
\end{example}

\begin{remark}
If $v$ is the periodic extension of a palindrome, as in Example~\ref{ex:3per} above (use the palindrome $(\frac 12, 2, \frac 12)$) and Example~\ref{ex:9per} below, then 
\[
\{M^{(j)}:j=0,\dots,K-1\}\ =\ \{\widetilde M^{(j)}:j=0,\dots,K-1\}, 
\]
so that, by Theorem~\ref{thm:M21_M11}\ref{it:thmTwoSided} and \ref{it:thmOneSided}, 
the FSM is applicable to $H$ if and only if it is applicable to $H_+$.
\end{remark}

\begin{example}\label{ex:5per}
The $5$-periodic Schrödinger operator $H$ with repeated $\frac 1{\sqrt 2}(1,1,0,1,0)$
has five monodromy matrices $M^{(0)}, M^{(1)}, M^{(2)}, M^{(3)}, M^{(4)}$ and five reversed order monodromy matrices $\widetilde{M}^{(0)}, \widetilde{M}^{(1)}, \widetilde{M}^{(2)}, \widetilde{M}^{(3)}, \widetilde{M}^{(4)}$.

All ten matrices have the trace $\frac 3{\sqrt 2}>2$, showing that $H$ is invertible, but the monodromy matrix
\begin{equation*}
M^{(0)}\ =\ \begin{pmatrix} -\frac1{\sqrt 2}& -1\\0 &-\sqrt 2\end{pmatrix}\ =\ \widetilde{M}^{(3)}
\end{equation*}
fulfils $M^{(0)}_{2,1} = 0$ and $| M^{(0)}_{1,1} | < 1$. Hence the FSM is not applicable, where one $L_+$
and one $R_-$ are not invertible. The other
matrices $M^{(j)}$ and $\widetilde{M}^{(j)}$ pass the test~\eqref{eq:checkM}, though.
So, as announced in Remark~\ref{rem:12}\ref{it:sharpExamples}, this operator $H$ is not FSM-simple.

\end{example}

Note that a similar $7$-periodic example was discussed in the Bachelor thesis~\cite{Weber.2018}, which initially led to this study of the periodic case in a more structured way.

\begin{example}\label{ex:9per}
Consider the $9$-periodic Schrödinger operator $H$ with repeatedly $\frac 12(1,1,0,1,0,1,0,1,1)$ as its potential. All monodromy matrices (reflected or not) have trace equal to $-\frac 52\not\in[-2,2]$, whence $H$ is invertible. Like in Example~\ref{ex:3per}, the symmetry of the finite word above leads to symmetry of $v$ and, if $v$ is accordingly shifted, to $M^{(j)}=\widetilde{M}^{(j)}$ for all $j$.

  The $(2,1)$-entry of $M^{(j)}$ is zero for $j=0$ and for $j=1$:
\begin{equation*}
M^{(0)}\ =\ \begin{pmatrix} -\frac 12& 0\\0 &-2\end{pmatrix}\ =\ \widetilde{M}^{(0)},\qquad
M^{(1)}\ =\ \begin{pmatrix} -2& -\frac 34\\0 &-\frac 12\end{pmatrix}\ =\ \widetilde{M}^{(1)}.
\end{equation*}
Furthermore, the $(1,1)$-entry of $M^{(0)}$ is less than one in modulus.
Consequently, the FSM does not apply to $H$, so that $H$ is not FSM-simple.
\end{example}

\begin{example}\label{ex:onesided_only}
Consider the $9$-periodic Schrödinger operator $H$ with repeatedly $\frac 1{\sqrt 2}(1, 1, 1, 0, 1, 1, 0, 1, 0)$ as its potential.
We readily check that 
\begin{equation*}
M^{(1)} = \ \begin{pmatrix} -\frac1{\sqrt 2}& 2\\0 &-\sqrt 2\end{pmatrix}\
\end{equation*}
and therefore, as in Example~\ref{ex:5per}, the FSM is not applicable to $H$.
However, in contrast to Example~\ref{ex:5per}, there is no $j\in\{0,\dots, K-1\}$ such that $\widetilde{M}^{(j)} = M^{(1)}$.
In fact, one can check that the conditions of Theorem~\ref{thm:M21_M11}\ref{it:thmOneSided} are satisfied, i.e.\ that the FSM is applicable to $H_+$.
For an extensive list of all matrices $M^{(j)}$ and $\widetilde{M}^{(j)}$ see~\cite{Gabel.2021}.

In this computation, we assume that the finite vector $\frac 1{\sqrt 2}(1, 1, 1, 0, 1, 1, 0, 1, 0)$ forms entries $v(0),\dots,v(8)$ of the potential. If we place it at $v(-1),\dots,v(7)$ then the corresponding $M^{(0)}$, instead of $M^{(1)}$, fails the test~\eqref{eq:checkM}, so that $H_+$ is not invertible, whence the FSM is not applicable to it.
\end{example}

\subsection{Systematic Studies of \texorpdfstring{$\{0,\lambda\}$}{\{0, lambda\}}-Valued Potentials}\label{sec:0,lambda-systematic}
In this section, we show an algorithm to find non-FSM-simple examples like Examples~\ref{ex:5per} and \ref{ex:9per}. And because our algorithm does, in a certain sense, a systematic search, it can even be used to prove a positive result (for certain periods $K$) by \emph{not} finding such examples.

Recall from Section~\ref{sec:0,lambda} that for $K\in\NN$ and $w\in\{0,1\}^K$, the periodic extension of $\lambda\cdot w$ is denoted by $v\in\{0,\lambda\}^\ZZ$ and the corresponding periodic Schrödinger operator by $H$.

Without loss of generality, assume $K\ge 2$ and compute the monodromy matrix $M$ by $M(0)$ in $\eqref{eq:TransferMDefinition}$.
 Now the four entries of $M$ are given by polynomials in $\lambda$ and, because of \eqref{eq:H+inv}, we are particularly interested in values $\lambda$ where $M_{2,1}=0$. 
Up to period $K=9$, we can compute these zeros $\lambda$ exactly by radicals for the following reason: 
the polynomial $M_{2,1}(\lambda)$ has degree at most $K$.
One can see, by induction over $K$, that $M_{2,1}(\lambda)$ is always entirely even or odd.
Hence, polynomials of degree $6$ or $8$ can be reduced by substitution $\mu=\lambda^2$ to degree $3$ or $4$. For degree $5$, $7$, and $9$, one can factor out one $\lambda$ and apply the substitution $\mu=\lambda^2$ to the rest. So one always ends up with a polynomial of degree at most $4$ and can solve in radicals by the standard formulas.
 
Let us demonstrate the procedure for period $K=3$ in Table~\ref{tab:K=3}.
\begin{table}[!ht]
\begin{center}
	{\def\arraystretch{1.5}\tabcolsep=1.3ex
		\begin{tabular}{c@{\hskip 3ex}c@{\hskip 3ex}c@{\hskip 3ex}cc}
			\toprule
			\rowcolor{gray!45}
			$w$ & $ M $ & zeros of $M_{2,1}$ & $ \tr(M) $  & at zeros
			\\
			\bottomrule
      $(0,0,0)$ & $
                  \left(\begin{smallmatrix}
                    0&1\\ -1&0
                  \end{smallmatrix}\right) 
                  $
                  & $\varnothing$ & $0$
			\\ \rowcolor{gray!25}
      $(1,0,0)$ & $
                  \left(\begin{smallmatrix}
                    \lambda&1\\-1&0
                  \end{smallmatrix}\right) 
                  $
                  & $\varnothing$ & $0$ &
			\\
      $(0,1,0)$ & $
                  \left(\begin{smallmatrix}
                    0&1\\ -1&\lambda
                  \end{smallmatrix}\right) 
                  $
                  & $\varnothing$ & $\lambda$
			\\ \rowcolor{gray!25}
      $(1,1,0)$ & $
                  \left(\begin{smallmatrix}
                    \lambda&1\\ \lambda^2-1&\lambda
                  \end{smallmatrix}\right) 
                  $
                  & $ \pm 1$ & $2\lambda$ & $\pm 2$
			\\
      $(0,0,1)$ & $
                  \left(\begin{smallmatrix}
                    \lambda&1\\ -1&0
                  \end{smallmatrix}\right) 
                  $
                  & $\varnothing$ & $\lambda$
			\\ \rowcolor{gray!25}
      $(1,0,1)$ & $
                  \left(\begin{smallmatrix}
                    2\lambda&1\\ -1&0
                  \end{smallmatrix}\right) 
                  $
                  & $\varnothing$ & $2\lambda$ &
      \\
      $(0,1,1)$ & $
                  \left(\begin{smallmatrix}
                    \lambda&-\lambda^2+1\\-1&\lambda
                  \end{smallmatrix}\right) 
                  $
                  & $\varnothing$ & $2\lambda$
			\\ \rowcolor{gray!25}
      $(1,1,1)$ & $
                  \left(\begin{smallmatrix}
                    -\lambda^3+2\lambda \hphantom{x} &  -\lambda^2+1\\ \lambda^2-1&\lambda
                  \end{smallmatrix}\right) 
                  $
                  & $ \pm 1$ & $-\lambda^3+3\lambda$ & $ \pm 2$
			\\
			[0.1cm]
			\bottomrule
		\end{tabular}
	}
\end{center}
\caption{Procedure for $K = 3$}\label{tab:K=3}
\end{table}

We conclude that, for $K=3$, all $w\in\{0,1\}^3$, and all $\lambda\in\RR$, the implication
\begin{equation}\label{eq:implK=3}
M_{2,1}=0\qquad\stackrel{\text{Table}~\ref{tab:K=3}}\Longrightarrow\qquad |\tr(M)|=2
\qquad\stackrel{\ref{prop:trace_condition1}}\Longrightarrow\qquad H\text{ is not invertible}
\end{equation}
holds. The contraposition of \eqref{eq:implK=3} then shows
\begin{equation}\label{eq:impl2.K=3}
H\text{ invertible}\qquad\stackrel{\eqref{eq:implK=3}}\Longrightarrow\qquad M_{2,1}\ne 0
\qquad\stackrel{\eqref{eq:H+inv}}\Longrightarrow\qquad H_+\text{ is invertible}.
\end{equation}
By Lemma~\ref{lem:limOpsPeriodic}\ref{it:limOpsPeriodic}, all $L\in\Lim_-(H)$ and all $\widetilde{L}\in\Lim_-(H^{\mathrm{R}})$ are again 3-periodic Schrödinger operators with $\{0,\lambda\}$-valued potential that are invertible if $H$ is invertible by Lemma~\ref{lem:fundamental}.
We apply the reasoning~\eqref{eq:impl2.K=3} with $L$ and $\widetilde{L}$ in place of $H$ and derive that all corresponding compressions $L_+$ and $\widetilde{L}_+$ are invertible if $H$ is invertible.

Summarising, we conclude by Proposition~\ref{prop:FSMCondition} that the FSM is applicable to $H$ if $H$ is invertible, i.e., $H$ is FSM-simple. 
This settles the case $K=3$ in Theorem~\ref{thm:integerFSM}\ref{it:thmRat}.

For $K=4$, the corresponding table is easily computed and shows the same implications. As an illustration, Figure~\ref{fig:K4} shows the $4$-periodic example given by $w = (1,1,0,1)$. 
In this case, we get
\begin{equation*}
              M = 
              \begin{pmatrix}
                -2\lambda^2 + 1 & -2\lambda \\
                \lambda         & 1
              \end{pmatrix},
\end{equation*}
so that $M_{2,1} = 0$ if and only if $\lambda = 0$. In this case, however $\tr(M) = 2$. 
For $w=(0,0,0,0)$, one gets 
$M=                  
\left(\begin{smallmatrix}
  1&0\\0&1
\end{smallmatrix}\right) 
$, so that $M_{2,1} = 0$ for all $\lambda\in\RR$, but also here, $\tr(M) = 2$. After checking the other 14 cases of $w\in\{0,1\}^4$ in the same manner, we get the same result as in \eqref{eq:implK=3}.

The number of cases for period $K$ is obviously $2^K$. Because of this exponential growth, we have written a code in \emph{SageMath}, cf.~\cite{Sage}. Note that actually all computations are purely symbolic; the zeros are computed exactly up to $K=9$ 
and all tables and the code are well-documented in~\cite{Gabel.2021}. An interim summary at this point could be:
\vskip1ex
\begin{addmargin}[2em]{2em}
  \noindent\emph{If $K\in\{1,2,3,4\}$ and $\lambda\in\RR$, then all $K$-periodic discrete Schrödinger operators with a $\{0,\lambda\}$-valued potential are FSM-simple.}
\end{addmargin}
\vskip1ex
Note that, for $K=5$, one finds that $H$ defined by $\lambda=\frac{1}{\sqrt2}$ and $w=(1,1,0,1,0)$ is not FSM-simple, as shown in Example~\ref{ex:5per} and Figure~\ref{fig:K5}. 
Also for $w=(1,1,0,1,1)$ and $w=(1,1,1,1,1)$,
the entry $M_{2,1}$ has irrational zeros with $|\tr(M)|>2$.
However, we find that the only non-FSM-simple examples with $K=5$ have irrational $\lambda$. 
This behaviour continues up to period $K=8$, proving the statement:
\vskip1ex
\begin{addmargin}[2em]{2em}
  \noindent\emph{If $K\in\{5,6,7,8\}$ and $\lambda\in\QQ$, then all  $K$-periodic discrete Schrödinger operators $H$ with a $\{0,\lambda\}$-valued potential are FSM-simple.}
\end{addmargin}
\vskip1ex
By comparison, we find Examples~\ref{ex:9per} and~\ref{ex:onesided_only} at period $K=9$.

\subsection{Proof of Theorem~\ref{thm:integerFSM}}
We restrict ourselves to the two-sided infinite case here and show that $H$ on $\ell^p(\ZZ)$ is FSM-simple in the respective situations. 
By Theorem~\ref{thm:M21_M11}, it follows that also the one-sided compression $H_+$ is FSM-simple.

For~\ref{it:thmInt}: let the potential be integer-valued, and suppose that $H$ is invertible.
We have to show that the FSM is applicable to $H$.
We prove applicability of the FSM via Proposition~\ref{prop:FSMCondition}. As $H$ is invertible, it remains to show that all $L_+$ with $L\in\Lim_-(H)$ and all $R_-$ with $R\in\Lim_+(H)$ are invertible.
We argue with the following three steps:
\begin{enumerate}[label=(\Alph*)]
  \item All $L\in\Lim_-(H)$ and all $R\in\Lim_+(H)$ are again periodic Schrödinger operators with integer-valued potential by Lemma~\ref{lem:limOpsPeriodic}\ref{it:limOpsPeriodic}, and all are invertible as a consequence of Lemma~\ref{lem:fundamental} since $H$ is invertible.
\item Applying Proposition~\ref{prop:oneSidedPeriodic} to all such $L$ and $R$ in place of $H$ gives that each $L_+$ is injective on $\ell^\infty(\ZZ_+)$ and each $R_-$ is injective on $\ell^\infty(\ZZ_-)$.
\item It remains to apply Corollary~\ref{cor:injImpliesInvertibleForLimOps} to every such $L$ and $R$, which are, if considered as acting on $\ell^2(\ZZ)$, self-adjoint invertible band operators, to see that all $L_+$ with $L\in\Lim_-(H)$ and all $R_-$ with $R\in\Lim_+(H)$ are invertible.
\end{enumerate}
This completes the proof of part~\ref{it:thmInt}.

For~\ref{it:thmRat}: see Section~\ref{sec:0,lambda-systematic}.

For~\ref{it:thmTwoPer}:
        let $H$ be a $2$-periodic Schrödinger operator with real potential $v$, not necessarily integer-valued.
        If $L\in\Lim_-(H)$ and $\widetilde{L}\in\Lim_-(H^{\mathrm{R}})$, then also $L$ and $\widetilde{L}$ are $2$-periodic as a consequence of Lemma~\ref{lem:limOpsPeriodic}\ref{it:limOpsPeriodic}.
By Example~\ref{ex:gapsPeriodic}\ref{it:gapsPeriodic-2},
with $L$ and $\widetilde{L}$ in place of $H$, we conclude
\begin{equation*}
  \sigma(L_+) = \sigma(L) = \sigma(H)\qquad\text{and}\qquad\sigma(\widetilde{L}_+) = \sigma(\widetilde{L}) = \sigma(H^{\mathrm{R}})= \sigma(H).
\end{equation*}
So if $H$ is invertible, all $L_+$ and $\widetilde{L}_+$ are invertible, too, and the FSM is applicable to $H$, by Proposition~\ref{prop:FSMCondition}.
\qed

For a $2$-periodic potential, Figure~\ref{fig:K2} shows a plot of the critical region outside of which the trace condition is fulfilled and the FSM is applicable to $H$.
    
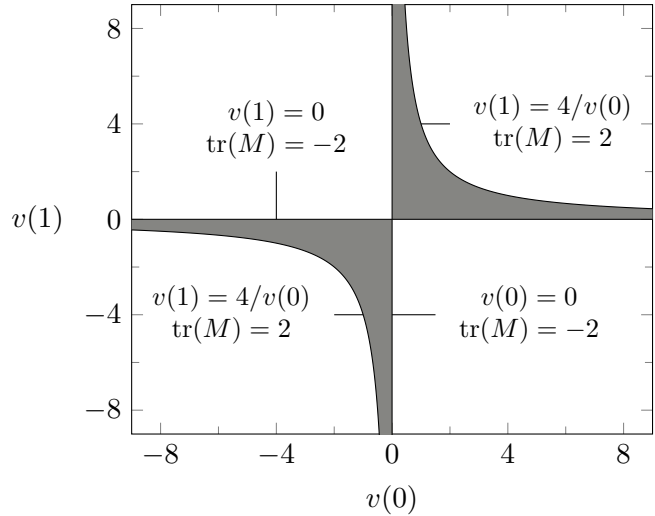
\begin{figure}[htbp]
    \begin{tikzpicture}
\begin{axis}[
xlabel = {$v(0)$} ,
ylabel = {$v(1)$} ,
ylabel style={rotate=-90} ,
xtick={-8,-4,0,4,8},
minor x tick num=1,
ytick={-8,-4,0,4,8},
minor y tick num=1,
xmin=-9 ,
xmax=9 ,
ymin=-9 ,
ymax=9,]
 \addplot [ mark=none, domain=-9:9,samples=901,name path=plot]
  {4/(x) };
 \path[name path=xaxis] (-9,0) -- (9,0);
 \path[name path=yaxis] (0,-9) -- (0,-9);
  \addplot[color=black] coordinates{(1,4) (2,4)} node[right] {\small\begin{tabular}{c} $v(1) = 4/v(0)$  \\ $\tr(M) = 2 $ \end{tabular}};
    \addplot[color=black] coordinates{(-1,-4) (-2,-4)} node[left] {\small\begin{tabular}{c} $v(1) = 4/v(0)$  \\ $\tr(M) = 2 $ \end{tabular}};
      \addplot[color=black] coordinates{(0,-4) (1.5,-4)} node[right] {\small\begin{tabular}{c} $v(0) = 0$  \\ $\tr(M) = -2 $ \end{tabular}};
        \addplot[color=black] coordinates{(-4,0) (-4,2)} node[above] {\small\begin{tabular}{c} $v(1) = 0$  \\ $\tr(M) = -2 $ \end{tabular}};
 \addplot[bandSpecCol] fill between [of=plot and xaxis];
 \draw (-9,0) -- (9,0);
\end{axis}
\end{tikzpicture}
   \caption{\footnotesize
  Example~\ref{ex:twosidedPeriodic}\ref{ex:twosidedPeriodic-2} for $E=0$ gives $\tr(M) = -2 + v(0)v(1)$. For all points $(v(0), v(1))$ lying outside of the closed grey region, we have $|\tr(M)|>2$. This means $v(0)v(1)\not\in[0,4]$ and, hence, the FSM for the $2$-periodic Schrödinger operator $H$ with potential $(v(0), v(1))$ is applicable. 
  }\label{fig:K2}  
  \end{figure}

\section*{Acknowledgement}
The authors would like to thank the anonymous referee for his or her interest and helpful comments on our manuscript.

\typeout{get arXiv to do 4 passes: Label(s) may have changed. Rerun}

\end{document}